\documentclass[12pt]{amsart}
\usepackage{amsmath,amssymb,fullpage,xy,amsthm,,graphicx,verbatim}
\usepackage{hyperref}
\usepackage{multirow}

\usepackage[mathcal]{euscript}
\newcommand{\Gal}{\operatorname{Gal}}

\newcommand{\p}{\mathfrak{p}}
\newcommand{\Fp}{\mathbf{F}_{\mathfrak{p}}}
\newcommand{\OK}{\mathcal{O}_{K}}

\newcommand{\Fl}{\mathbf{F}_{\ell}}

\newcommand{\Aut}{\operatorname{Aut}}

\newcommand{\im}{\operatorname{im}}

\newcommand{\Z}{\mathbb{Z}}
\newcommand{\F}{\mathbb{F}}
\newcommand{\FF}{\mathbb{F}}

\newcommand{\Q}{\mathbb{Q}}

\newcommand{\OO}{\operatorname{O}}
\newcommand{\SO}{\operatorname{SO}}
\newcommand{\GSp}{\operatorname{GSp}}
\newcommand{\PSL}{\operatorname{PSL}}

\newcommand{\PGL}{\operatorname{PGL}}
\newcommand{\PSU}{\operatorname{PSU}}
\newcommand{\SU}{\operatorname{SU}}

\newcommand{\Sp}{\operatorname{Sp}}
\newcommand{\GL}{\operatorname{GL}}
\newcommand{\SL}{\operatorname{SL}}

\newcommand{\Sym}{\operatorname{Sym}}

\newcommand{\C}{\mathbb{C}}

\newcommand{\GG}{\mathbf{G}}

\newcommand{\St}{\mathsf{St}}

\newcommand{\tors}{\mathsf{tors}}

\newcommand{\AGL}{\operatorname{AGL}}
\newcommand{\AGGL}{\operatorname{A\Gamma L}}
\newcommand{\ASL}{\operatorname{ASL}}

\newcommand{\bp}{\begin{proof}}
\newcommand{\enp}{\end{proof}}

\theoremstyle{plain}
\newtheorem{thm}[equation]{Theorem}
\newtheorem{lem}[equation]{Lemma}

\newtheorem{cor}[equation]{Corollary}
\newtheorem{prop}[equation]{Proposition}

\newtheorem{question}[equation]{Question}

\theoremstyle{remark}
\newtheorem{rmk}[equation]{Remark}

\def\d12{{_{12}}}

\def\acf{{algebraically closed field }}

\newcommand{\up}{^{-1} }
\newcommand{\om}{\omega }
\newcommand{\si}{\sigma }
\newcommand{\al}{\alpha }
\newcommand{\be}{\beta }
\newcommand{\med}{\medskip }
\newcommand{\lam}{\lambda}
\newcommand{\ep}{\varepsilon}
\newcommand{\ld}{,\dots,}
\newcommand{\lan}{ \langle }
\newcommand{\ran}{ \rangle }
\newcommand{\Id}{\mathop{\rm Id}\nolimits}
\def\ei{{eigenvalue }}

\def\ei{{eigenvalue }}

\def\f{{following }}

\def\ii{{if and only if }}

\def\ir{{irreducible }}

\def\irr{{irreducible representation }}

\def\itf{{It follows that }}

\def\mult{{multiplicity }}

\def\hw{{highest weight }}

\def\rep{{representation }}
\def\reps{{representations }}

\newcommand{\diag}{\mathop{\rm diag}\nolimits}
\newcommand{\bl}{\begin{lem}\label}
\newcommand{\el}{\end{lem}}
\numberwithin{equation}{section}

 \input xy
 \xyoption{all}

\begin{document}

\title{Unisingular Representations in Arithmetic and Lie Theory}

\author{John Cullinan}
\address{Department of Mathematics, Bard College, Annandale-On-Hudson, NY 12504, USA}
\email{cullinan@bard.edu}
\urladdr{\url{http://faculty.bard.edu/cullinan/}}

\author{Alexandre Zalesski}
\address{Department of Physics, Mathematics and Informatics, National Academy of Sciences of Belarus,
66 Prospekt  Nezavisimosti, Minsk, Belarus}
\email{alexandre.zalesski@gmail.com}
%\urladdr{\url{url address}}

\begin{abstract}
Let $G$ be a subgroup of $\GL(V)$, where $V$ is finite dimensional vector space over a finite field of characteristic $p >0$.  If $\det(g-1) = 0$ for all $g \in G$ then we call $G$ a \emph{fixed-point subgroup} of $\GL(V)$. Motivated in parallel by questions in arithmetic and linear group theory, we  classify all irreducible fixed-point subgroups of $\Sp_8(2)$ and give new infinite series of irreducible fixed-point subgroups of symplectic groups $\Sp_m(2)$ for various $m$ arising from certain representations of groups of Lie type.
% results on the Steinberg representations of certain finite groups of Lie type.
\end{abstract}

\maketitle

\begin{center}
\emph{Dedicated to the memory of James Humphreys} \\
\end{center}

\section{Introduction}

Let $G$ be a finite group. If
$$
\varphi: G \to \GL(V)
$$
is a representation of $G$ such that $\det(\varphi(g) - 1) =0$ for all $g \in G$, then we call $\varphi$ \emph{unisingular} and $\varphi(G)$ a \emph{fixed-point subgroup} of $\GL(V)$.  Unisingular representations are of significant interest for group \rep theory (see, for example, \cite{GT}, \cite{z09}, \cite{z18} and \cite{z20})
and also arise naturally in arithmetic in counting points on varieties over finite fields.  More precisely, we are  motivated by the following setup that we now describe in detail.

Fix positive integers $m \geq 2$ and $d \geq 1$.  Let $A$ be an abelian variety of dimension $d$ defined over a number field $K$ and let $\p \subseteq \OK$ denote a nonzero prime of good reduction for $A$.  We write $A(K)$ for the Mordell-Weil group of $A$ and, by abuse of notation, $A(\Fp)$ for the $\Fp$-points of $A$ mod $\p$, where $\Fp$ denotes the residue field $\OK/\p$.  It is well-known that $A(K)$ is a finitely-generated abelian group and that $A(\Fp)$ is a finite abelian group.  Moreover, given any $P \in A(K)$ the reduction-mod-$\p$ map $\pi_\p$ produces a point $\pi_\p(P) \in A(\Fp)$.  However, the map $\pi_\p$ is not necessarily surjective.

Since $A(K)$ is finitely-generated, we may write $A(K) = A(K)_{\tors} \times A(K)_{\textsf{free}}$ as the decomposition into its torsion and free subgroups.  Suppose additionally that $A(K)_{\tors}$ has a subgroup $S$ of order $m$.  Then for all $\p$ of residue characteristic coprime to $m$, $A(\Fp)$ will have a subgroup of order $m$ as well, since $S$ will survive modulo such primes (see \cite[Appendix]{katz} for background on reduction modulo $\p$ and proofs of these assertions). Thus, we have an easy ``global-to-local'' principle: \\

\begin{quote}\emph{If the order of $A(K)_\tors$ is divisible by $m$, then the order of $A(\Fp)$ is divisible by $m$ for all but finitely many $\p \subseteq \OK$.}
\end{quote}

The interesting question is whether there exists a converse local-to-global principle.  Interpreted literally, it is easy to come up with examples of elliptic curves $E$ over $\Q$ that have a point of order $m$ modulo all but finitely many primes $p$, but do not have a point of over $m$ defined over $\Q$.  For example, the elliptic curve \textsf{11.a1} of the \textsf{LMFDB} \cite{lmfdb} has trivial Mordell-Weil group over $\Q$ but has a point of order 5 modulo every $p \ne 11$. The more refined local-to-global question, originally posed by Lang, is the following.

\begin{question} \label{lq}
Fix a positive integer $m \geq 2$ and let $A$ be an abelian variety defined over a number field $K$.  Suppose for all but finitely many primes $\p$ of $K$ the orders of the groups $A(\Fp)$ are all divisible by $m$.  Does there exist a $K$-isogenous $A'$ such that $m$ divides the order of $A'(K)_\tors$?
\end{question}

Lang's question reflects the fact if $A/K$ is an abelian variety such that $A(\Fp)$ has a subgroup of order $m$, then so does every member $A'$ of the $K$-isogeny class $I_{K,A}$ of $A$.  In \cite{katz}, Katz showed that the answer to this Question \ref{lq} -- \textsf{Yes} or \textsf{No} -- can be studied entirely via group theory.  In the special case where $m$ is a prime number, it amounts to classifying certain unisingular representations.  When $m$ is composite, the group-theoretic reformulation is more complicated, and we do not address it in this paper, but it is still group theory nonetheless.

Before transitioning entirely to group theory for the remainder of the paper, we review what is known from the point of view of abelian varieties.  In the table below we classify our answers of \textsf{Yes} or \textsf{No} by both the dimension $d$ of the abelian variety $A$ and the modulus $m$.  An answer of \textsf{Yes} for  fixed   $d$ and   $m$ means that, for every abelian variety $A$ of dimension $d$ defined over any number field $K$ for which   $A_\p(\Fp) $ has a subgroup  of order $m$, %\equiv 0 \pmod{m}$,
the $K$-isogeny class of $A$ contains an abelian variety $A'$ (possibly $A$ itself) such that $A'$ has a $K$-rational subgroup of order $m$. An answer of \textsf{No} means that there exists a number field $K$ and a $K$-isogeny class $I$ of abelian varieties defined over $K$ such that $m$ divides the order of $A_\p(\Fp) $ for every element $A$ of $I$, but no element of $I$ has a $K$-rational subgroup of order $m$.  For more details on the arithmetic consequences of the following table, see \cite[Introduction]{cy}.

\med
\begin{center}
\begin{tabular}{|c|c|l|c|}
\hline
$\dim A$ & answer & modulus  $m$ & Reference\\
\hline
1 & \textsf{Yes}  & all $m \geq 2$ &\cite{katz} \\
\hline
2 & \textsf{Yes}  & all prime $m \geq 2$ & \cite{katz} \\
& \textsf{No}  & all prime powers $m = \ell^n \geq 4$ & \cite{cy} \\
\hline
3 & \textsf{Yes} & $m=2$ & \cite{2tors} \\
& \textsf{No} & all $m\geq 3$ & \cite{katz} \\
\hline
$\geq 4$ & \textsf{No} & all $m\geq 2$ & \cite{2tors, katz}\\
\hline
\end{tabular}
\end{center}
\med

 It is the answers of \textsf{No} that interest us, both from the point of view of number theory and group theory.   From the number theory perspective, it would be interesting to \emph{construct} the abelian varieties over small-degree number fields that violate this local-to-global principle, since they exhibit non-generic properties.  We remark that all known instances of the answer \textsf{No} are non-constructive and are existence proofs only.  See Corollary \ref{avcor} below for an instance of this.

We highlight one of these cases as the main motivating example for this paper.  In arithmetic language, the answer \textsf{No} in dimension 4 and modulus 2 means that there exists, for example, an abelian fourfold over a number field that has an even number of points modulo all but finitely many $\p$, but does not have a global point of order 2.  This is in contrast to, say, the same setup for elliptic curves (\emph{i.e.}~$\dim A = 1$) where it is an easy exercise in Galois theory to show that the answer is \textsf{Yes} for the modulus $m=2$.

We now turn to the group theoretic perspective and present the version of Katz' reformulation for prime moduli that motivates us; see \cite{katz} for the details of the reformulation.

\begin{question} \label{kq} \cite[Problem II, p.483]{katz}
Let $\ell \geq 2$ be prime number.  Let $\overline{\rho}_\ell:\Gal(\overline{K}/K) \to \Aut(T_\ell A \otimes \Fl)$ be the mod $\ell$ representation.
 %of $A$.
  Suppose $\det(g-1) =0$ for all $g \in \Gal(\overline{K}/K)$.  Does the composition series of $\overline{\rho}_\ell$ contain the trivial representation?
\end{question}

Upon choosing a basis for $T_\ell A \otimes \Fl$ and invoking the Galois-equivariance of the symplectic Weil pairing on $T_\ell A$, one obtains a purely group-theoretic version of Question \ref{kq} (see \cite[Introduction]{cy} for details).  We switch from $\ell$ to $p$ for the underlying prime modulus to coincide with more standard notation from finite group theory.  This is the now the question we take up in this paper.

\begin{question}\label{czq}
Let $p$ be a prime number and $G$ a subgroup of $\Sp_{2d}(p)$.  Suppose that  $\det(g-1) = 0$ for all $g \in G$, that is, G is fixed point.  Does the Jordan-H\"older series of $\F_pG$ contain the trivial representation?
\end{question}

We note that Question \ref{czq} is equivalent to Question \ref{lq} when $m=p$ is prime and $d=\dim A$ in the sense that one has a positive answer if and only if the other does.   We now recall what is known from the point of view of group theory.

In \cite{katz}, Katz showed that when $d=1$ and $d=2$, the answer to Question \ref{czq} is \textsf{Yes}
for every group $G$. For every $d \geq 3$ and $p$  odd  he also provided  examples of  groups $G$ with answer \textsf{No}.  % by constructing an explicit unisingular \rep $G\rightarrow \Sp_{2d}(p)$ for every $d>3$ for $G$ an elementary abelian group of order 8.  (For $d=3$ one has a similar \rep with $|G|=4$.)
In \cite{sp6} and \cite{smallchar} the first-named author classified the groups $G \subset \Sp_6(p)$ for which the answer to Question \ref{czq} is \textsf{No};   all such groups $G$ act reducibly on the underlying symplectic space.  In \cite{2tors} he further showed that the answer to Question \ref{czq} is \textsf{Yes} when $d=3$ and $p=2$ (for every group) and \textsf{No} when $d \geq 4$ and $p=2$ for some $G$.  Therefore, the rough answer -- \textsf{Yes} or \textsf{No} -- to Question \ref{czq} is now known for every $d \geq 1$ and every prime number $p$.

The problem remains to classify the groups $G$ giving an answer of \textsf{No}.    This is an interesting problem from the point of view of group theory and representation theory (see \cite{z18,GT,z20}, for example) and from the point of view of arithmetic as well.  In terms of arithmetic, an irreducible unisingular symplectic representation of degree $2d$ is attached to a simple abelian variety of dimension $d$.  The irreducibility of the representation means the answer to Question \ref{czq} is \textsf{No} and that there is a simple abelian variety $A$ of dimension $d$ for which the answer to Question \ref{lq} is \textsf{No} for the modulus $m=p$.

To discuss the classification problem stated above we first mention the  \f trivial fact:

 \bl{tf1} Let $V_i$ $(i=1,2)$ be a symplectic spaces and $G_i\subset \Sp(V_i)$ subgroups
 with no trivial composition factor.  Let $V=V_1\oplus V_2$
 be a symplectic space such that $V_1,V_2$ are mutually orthogonal. Suppose that $G_1$ is a fixed-point subgroup of $\Sp(V_1)$. Then $G=G_1\times G_2$ is a fixed-point subgroup of $\Sp(V)$ with no trivial composition factor.
\el

As $G_2$ does not need to be a fixed-point group, Lemma \ref{tf1} shows that the problem of obtaining a full classification of fixed-point subgroups $G\subset \Sp(V)$ for arbitrary $V$ is intractable. To avoid this difficulty %arising from  Lemma \ref{tf1}
 it seems to be reasonable to turn to classification of {\it maximal}\,
fixed-point subgroups (with no trivial composition factor). The bulk of this problem is then to classify
the \ir maximal fixed-point subgroups. Note however that no reduction of the general case to irreducible groups is available.

In fact, Katz \cite[pp. $500-501$]{katz} provided his \textsf{No} example of a fixed-point abelian subgroup of $\Sp_{2d}(p)$, $p$ odd, $d\geq 3$ without a trivial factor
on the basis of Lemma \ref{tf1} without stating it explicitly. In his example $|G_1|=4$ and $|G_2|=2$.  In the case of $p=2$ we give a similar example with $|G_1|=9,|G_2|=3$ (Corollary \ref{se1}). % This   contextualizes our example for $p=2$ alongside Katz' original counterexample for $p$ odd.

%\begin{rmk}
Let $q$ be a prime power.  One can also ask more generally for a classification of all fixed-point groups $G \subseteq \GL_d(q)$ as $d$ and $q$ vary, not just restricting to symplectic groups or prime fields.  Toward this end, for $d=2$  such groups are always reducible, see  \cite{katz} or \cite[Ch.I, \S 1, Exercise 1]{serre}.  For $d=3$ the full classification of fixed-point subgroups of $\GL_3(q)$, where $q$ is any prime power, appears in \cite{gl3}. Note that the group $\SO_{3}(q)$
  is an irreducible fixed-point subgroup of $\GL_3(q)$ whenever $q$ is odd, but when $q$ is even, every fixed-point subgroup of $\GL_3(q)$ is reducible. More generally, if $d$ is odd then $\SO_{d}(q)$
  is a fixed point \ir subgroup of $\GL_d(q)$. (This is well known, see for instance \cite[Lemma 2.27(1)]{VZ} for a straightforward proof; more conceptually, this also follows from the fact that the natural \rep of  the simple algebraic group of type $B_n$, $n\geq 1,$ has weight zero.  This final observation is in line with the work we do in Section 6 below.)
%\end{rmk}

%\jc{add motivation from Lie Theory} \\
\med
Duly motivated, we now present our main theorems.  First, we determine all  irreducible, fixed-point subgroups of $\Sp_8(2)$.

\begin{thm} \label{cmain}
Let $G$ be a maximal  irreducible fixed-point subgroup of $\Sp_8(2)$.  Then $G$ is conjugate to   %following list of groups:
$ {\rm L}_3(2):2$ or to ${\rm AGL_2}(3)\cong \PSU_3(2):S_3$. %or to  %and all isomorphic subgroups
  %{\rm A}\Gamma{\rm L}_1(9), \text{ or }
 % {\rm L}_3(2):2.
Furthermore, the \ir subgroups of these groups are $ {\rm L}_3(2)$ in $ {\rm L}_3(2):2$
and $\PSU_3(2)$, $\PSU_3(2):2$, $\PSU_3(2):3$, and $\AGL_1(9)$ in $\AGL_2(3)$.\end{thm}

We observe that these groups are absolutely irreducible.  This then gives us a classification of the images of the mod-2 representations on 4-dimensional absolutely simple abelian varieties for which the answer to Question \ref{lq} is \textsf{No} for the modulus 2.  We also remark that in terms of \emph{constructing} such a fourfold, there is at least the possibility that the base field can be taken to be $\Q$; for general $p$-torsion one needs the number field $K$ of definition of $A$ to contain $\Q(\zeta_p)$ in order for the image of $\overline{\rho}_p$ to lie in $\Sp_8(p)$ as opposed to $\GSp_8(p)$. This remark adds to our interest in the special case of 2-torsion.

We discover that the groups of Theorem \ref{cmain} are the first instances of  several infinite families of absolutely irreducible fixed-point subgroups of $\Sp_m(2)$ for certain $m \geq 8$.  None of these families have been previously known.  The first such family is given in the following theorem.

\begin{thm}\label{af1} Let $q$ be odd and $m=q^n-1$. Suppose that $n>1$ or $n=1$ and q is not a prime.
Then there exists an absolutely \ir fixed-point subgroup of $\Sp_{m}(2)$ isomorphic to $\AGL_n(q)$.
\end{thm}

%\begin{rmk}
Note that the group $\AGL_n(q)$ with $q=r^k$, $r$ a prime, is a subgroup of $\AGL_{nk}(r)$. We expect that the latter group is a maximal fixed-point subgroup of $\Sp_{m}(2)$ for $m=q^n-1$.%\end{rmk}

The second family is exemplified by the isomorphism ${\rm L}_3(2):2\cong \PGL_2(7)$, and is given in the following theorem.

\begin{thm}\label{l2q} Let  $G=\PGL_2(q)$, where $4|(q+1)$ and $3|(q-1)$. Then $G$ has an absolutely \ir
unisingular \rep into $\Sp_{q+1}(2)$.\end{thm}

The third family arises from the Steinberg $2$-modular \rep of simple groups of Lie type, as follows.

\begin{thm} \label{zmain}\label{st1}
Let $G$ be a finite simple group of Lie type in characteristic $2$ and $\St_2$ the $2$-modular Steinberg \rep of $G$. Suppose that $G$ is not of type $A_1(q)$, ${}^2A_n(q)$ ($n$ odd), or $E_7(q)$. Then
$\St_2(G)$ is a fixed point subgroup of $\Sp_m(2)$, where $m=|G|_2$ is the $2$-part of the order of G.
\end{thm}

Note that a proof of this theorem is required only for groups of type %${}^2A_n(q)$ and
$A_n(q)$ with $n$ odd, $C_n(q)$ with $n\equiv 1,2\pmod 4$ and of type $D_n^\pm(q)$, $n\equiv 1,2\pmod 4$, as for the other groups the result has been obtained in an earlier paper of the second-named author; see Theorem \ref{z90} below. The group ${\rm L}_3(2)\cong A_2(2)$ arising in Theorem \ref{cmain}  for $m=8$ is the minimal order example of those in Theorem \ref{st1}.

We do not have a result as above for unitary groups ${}^2A_n(q)\cong \PSU_{n+1}(q)$ with $n$ odd and
for $E_7(q)$. For $G=A_1(q)$ the 2-modular Steinberg \rep does not yield a fixed-point subgroup (Lemma \ref{sl2}).

In terms of the original questions of Lang and Katz, we obtain the following Corollary for abelian varieties of certain dimensions $\geq 5$.

\begin{cor}\label{avcor}
Let $G \subset \Sp_m(2)$ be the symplectic embedding of any of the groups of Theorems {\rm \ref{af1}, \ref{l2q}}, or {\rm \ref{st1}}.  Then $G$ occurs as the image of the mod $2$ representation of an absolutely simple abelian variety of dimension $m/2$, defined over some number field $K$, such that $A_\p(\Fp)$ has an even number of points for every good prime $\p \subset \OK$, but no member of the $K$-isogeny class of $A$ has a global point of order $2$.\end{cor}

%\begin{rmk}
Corollary  \ref{avcor}, who's proof appears at the end of the paper, immediately before Appendix \ref{appendix}, shows that symplectic, irreducible, unisingular representations occur as Galois representations attached to abelian varieties $A$ over number fields $K$.  But the degree $[K:\Q]$ of the field of definition of $A$ may not be minimal if we use the purely Galois-theoretic construction of $A$.  It would be interesting to search for examples of these abelian varieties such that $[K:\Q]$ is minimized, or even to determine necessary lower bounds on $[K:\Q]$.
%\end{rmk}

For an \irr of $G\in\{\Sp_{2n}(q), \ {\rm Spin}^\pm_{2n}(q)$, $\SL_{n}(q)$ ($q$ even$) \}$, we give a rather strong sufficient condition for an \irr to be unisingular.

\begin{thm}
\label{s21}
Let $G$ be a finite simple group of Lie type $A_n(q)$, $C_n(q)$ or $D_n^\pm (q)$, q even,
$\mathbf{G}$ the respective simple algebraic group and let $\om_1\ld \om_n$ be the fundamental weights of the weight system of $\mathbf{G}$. Let $V$ be an \ir $\mathbf{G}$-module. Suppose that the \f conditions hold:

%\begin{enumerate}
$(1)$ $\mathbf{G}$ is of type $A_n$ and there  are natural numbers $m_1,m_2$ and $i\in\{1\ld n\}$ such that $m_1(q-1)\om_i+\om_1+\om_n$ and $m_2(q-1)\om_i$ are weights of $V$.

$(2)$ $\mathbf{G}$ is of type $C_n$ $(n>1)$ or $D_n$ $(n>3)$ and there are natural numbers $m_1,m_2,m_3$ such that  %Let V be an \ir $m_1\mathbf{G}$-module with \hw $\om$.
$m_1(q+1)\om_1, m_2(q-1)\om_1$ and $m_3(q-1)\om_1+\om_2$ are weights of V.
%\end{enumerate}

\noindent Then  every semisimple element $g\in G$ has \ei $1$ on V.\end{thm}

We show that this condition is satisfied by the $p$-modular Steinberg \rep of $G$.
%(Note that $i=1,n$ are not excluded in (1).)

Finally, a natural question would be to extend our fixed-point subgroup result in Theorem \ref{cmain} to larger $n$ ($n=10,12$, etc.).  However, this would lead to an analysis of many cases which do not depend uniformly on $n$. For this reason we think that further work on understanding the unisingularity problem in higher rank must focus on constructing  unisingular representations of uniform families of groups.
 This paper gives a certain contribution for simple groups of Lie type.

Obtaining necessary and sufficient condition of unisingularity for an \irr of arbitrary simple groups of
 Lie type in defining characteristic 2 does not seem to be a realistic task. However, this has been achieved for groups $\SL_n(2)$ and $\Sp_{2n}(2)$ in the papers \cite{z18,z20} of the second-named author.

\med
In the next two sections we review our notation and conventions.  We then focus on affine groups before the special case of $\Sp_8(2)$ where our goal is to classify all 8-dimensional symplectic irreducible fixed-point subgroups.  We conclude the paper by investigating groups of Lie type in more generality and add a brief appendix where we record the computational aspects of our classification. \\

\section{Background: Notation and conventions}

We denote by $\Z$, $\Q$, $\C$, $\F_q$ the set of integers, the field of rational and complex numbers,
and the finite field of $q$ elements, respectively. For a set $X$ we write $|X|$ for the cardinality of $X$.

\subsection{Finite Groups}

In our work below on finite groups, we use the symbols $S_n$, $A_n$, $D_{2n}$, and $C_n$ to denote the symmetric, alternating, dihedral, and cyclic groups on $n$ letters.  If $N$ and $Q$ are groups we use the notation $N.Q$ to denote any group with kernel $N$ and quotient $Q$.  If the extension is split, we use $N:Q$.  We use $N \wr_\phi Q$ to denote the wreath product $N^n \rtimes_\phi Q$, where $\phi:Q \to S_n$ is a permutation representation.  When the representation is understood from context we omit $\phi$ from the notation.

We import notation from the theory of finite and classical groups as well, such as ${\rm L}_n(q)$ for the simple group $\PSL_n(\F_q)$ over the field of $q$ elements.  In general, we follow the conventions of \cite{bhrd} for the finite classical groups.  While most of this notation is standard across the literature, we are careful to point out that the definition of the special orthogonal group tends to vary across the literature. We use notation   $\SO_8^\pm(2)$ for the special orthogonal group of degree 8,
and $\Omega^\pm_8(2)$ for its simple subgroup of index 2.

If $G$ is a finite group and $p$ is a prime number then we write $|G|$ for the order of $|G|$ and $|G|_{p}$ for the $p$-part of $|G|$, namely the greatest $p$-power dividing $|G|$.  Furthermore, if $G$ is a finite group of Lie type in defining characteristic $p$,  then there exists a unique irreducible representation of $G$ over the complex numbers $\C$ known as the \emph{Steinberg} representation of $G$.  Let $F$ be an algebraically closed field of characteristic $p$. Then there exists an irreducible $F$-representation of $G$ whose Brauer character coincides with the character of the Steinberg representation at the semisimple elements of $G$.  We use the notation $\St$ or $\St_G$ for the Steinberg representation of $G$ over $\C$, and $\St_q$ for the modular Steinberg representation over an algebraically closed field of characteristic $p|q$.

Because we will often focus on a specific prime (usually $p=2$), we follow standard conventions in groups theory and write $O_p(G)$ for the maximal normal $p$-subgroup of $G$. When $p$ is understood, we call a $p'$-element of $G$ one that has order coprime to $p$; a $p'$-group has order coprime to $p$.

If $V$ is a set and $Y \subset \Sym(V)$ then we write $V^Y$ for the set of fixed points
 of the action of $Y$ on $V$. If $V$ is a vector space and  $Y\subset \GL(V)$ then $V^Y$ is a subspace of $V$, and $Y$ is a fixed point set \ii $V^y\neq 0$ for every $y\in Y$.   If $V$ is a $KG$-module (where $K$ is a field) and $H\subset G$ a subgroup then $V|_H$ is the restriction of $V$ to $H$.

\subsection{Linear Algebraic Groups} We take this opportunity to set our conventions for the remainder of the paper and acknowledge that there are places where the standards of the finite group theory and Lie theory community overlap -- for example the symbol $A_n$ is standard notation for both the alternating group on $n$ letters and for one of the simple, simply connected algebraic groups of rank $n$.   However, it should be clear from the context what is meant.

Recall that simple simply connected algebraic groups partition to types denoted by $A_n$, $B_n$,  $C_n$, $D_n$, $E_6$, $E_7$, $E_8$, $F_4$, $G_2$ which naturally correspond to types of simple Lie algebras over $\C$. The subscript is called the rank of the group and Lie algebra. To every simple Lie algebra of rank $n$ corresponds a weight lattice $\Lambda$, which is a $\Z$-lattice of rank $n$ whose elements are called weights. Fixing a basis one can write every weight as $(z_1, \dots, z_n)$, where $z_1, \dots,  z_n \in \Z$. The weights
$\omega_i:=(0, \dots,  0,1,0, \dots, 0)$, $i=1, \dots,  n$, are called fundamental weights.  The weights $(z_1, \dots,  z_n)$ with non-negative entries are called dominant and we write $\Lambda^+$ for the set of dominant weights. The Lie algebra determines subsets of $\al_1\ld \al_n\in\Lambda$ called simple roots. The expressions of them in terms of $\om_1\ld \om_n$ is given in \cite[Planchees I-IX]{Bo}. Denote by $\mathcal{R}$  the sublattice of all $\Z$-linear combinations of simple roots and by  $\mathcal{R}^+$ the subset of $\Z$-linear combinations with non-negative coefficients.
The lattice $\Lambda$ is endowed with a partial ordering, which we denote by $\succeq$.
Specifically, $\lam\succeq\mu$ \ii $\lam-\mu\in\mathcal{R}^+$.

The same weight system is assigned to the respective simple algebraic group $\mathbf{G}$. The irreducible representations of $\mathbf{G}$ are parameterized by the dominant weights.  %Finally, we denote the associated root lattice by $\mathcal{R}$ and
We set $\si_q:=(q-1)(\om_1+\cdots+\om_n)$, where $q$ is a prime power.

\section{Preliminary and Known Results}

In this short section we review what is known and provide some additional results that will be of use in the proofs of the main theorems below.

\bl{tt1}%\label{tt1}
Let F be a field of characteristic $2$ and G a finite group. Let $\be$ be the Brauer character of a non-trivial absolutely \ir $FG$-module V. Suppose that $\be(g)$ is an integer for every odd order element $g\in G$. Then    $\rho(G)$ is equivalent to a \rep into $\Sp_d(2)$.  \el

\bp If $\be(g)\in\Z$ then $\be(g)\pmod 2\in\F_2$ and $\rho$ is  equivalent to a \rep into $\GL_d(2)$ by \cite[Ch. I, Theorem 19.3]{Fe}. In addition,
$V$ is self-dual  \cite[Ch. IV, Lemma 2.1(iv)]{Fe}, and hence $G$ preserves a non-degenerate alternating form on $V$ \cite[Ch. IV, Theorem 11.1 and Corollary 11.2]{Fe}. \enp
%Conversely, if $\rho(G)\subset \GL_d(2)$ then $\be(g)$ belongs to $\Q(\zeta)$, where $\zeta$ is a $|g|$-root of unity. As $|g|$ is odd, $\be(g)\pmod 2\in\F_2$ \ii $\be(g)$ is an integer. %[[John, maybe you can explain this better.]]
% By \cite[Ch. IV, Theorem 11.1 and Corollary 11.2]{Fe}, $\rho(G)$ preserves a non-degenerate alternating form on $V$ and  $d$ is even.

\bp[Proof of Theorem {\rm \ref{l2q}}] Suppose first that $G={\rm L}_2(q)$.
As $4|(q+1)$, the integer $(q-1)/2$ is odd and $q\geq 7$. Note that every ordinary \ir character of $G$ of degree $q+1$ is induced from a non-trivial one-dimensional character $\tau$ of $B$, a Borel subgroup of $G$, see \cite[\S 3]{spr} or elsewhere.  Then $\tau(b)$ for every $b\in B$ is a $(q-1)/2$-root of unity (as $|B|=q(q-1)/2$). In fact, $\tau^G$ is \ir whenever $\tau\neq 1_B$ (loc. cit.). As $|G|_2$ divides $q+1$, every character $\tau^G$ is of defect 0, and hence $\tau^G$ restricted to the odd order elements of $G$ is an \ir Brauer character. By the character table of $G$ we have $\tau^G(g)\in\Z$ for  $g\in G$ with $|g|$ odd unless $|g|$ divides $(q-1)/2$ and hence  $g$ is conjugate to an element of $B$. Let $g\in B$. Then $\tau^G(g)=\tau(g)+\tau(g)\up$. This is an integer \ii $\tau(g)^3=1$. As $\tau$ is non-trivial, this requires $3|(q-1)$, and if this holds then there is $\tau\neq 1_B$ with $\tau(g)^3=1$ for every $b\in B$. By Lemma \ref{tt1}, for this $\tau$ the Brauer character $\tau^G$  is the character of a \rep  $G\rightarrow \Sp_{q+1}(2)$.

Let $G=\PGL_2(q)$. Note that every odd order element  $g\in G$ lies in ${\rm L}_2(q)$, but $C_G(g)$ does not.
\itf that $G$ and $L_2(q)$ have the same number of conjugacy classes of odd order elements, and hence the same numbers of  2-modular \ir representations.  By Clifford's theorem, distinct \ir \reps of $G$ have no common \ir constituent under restriction to ${\rm L}_2(q)$, and hence every 2-modular \irr of $G$ is \ir on ${\rm L}_2(q)$. As the fixed point property is to be examined at the odd order elements only, the result follows from that for ${\rm L}_2(q)$.
\enp

We will make use of the following results extensively in later sections of the paper.

\begin{lem}\label{st2} {\rm \cite[\S 3.1 and \S 3.7]{Hu1}}
Let ${\mathbf G}$ be a simple simply connected algebraic group in defining characteristic $p$, let $\mathbf {G}_{\C}$ be a simple simply connected algebraic group of the same type as ${\mathbf G}$ and $L$ the Lie algebra of $\mathbf {G}_{\C}$. Let $\si_q:=(q-1,\dots,q-1)$ be an element of the weight system of ${\mathbf G}$, and let  $\Phi_q$ be an irreducible representation of ${\mathbf G}$ with highest weight $\si_q$.

 Then the weights of $\Phi_q$ are the same as the weights of an irreducible representation of the Lie algebra $L$ of the same type as ${\mathbf G}$ with highest weight $\si_q$.
\end{lem}

\begin{lem}\label{mm2}
Let $\lam,\lam',\lam-\lam'\in \Lambda^+$ and $\lam-\lam'\in \mathcal{R}.$ Then $\lam\succeq\lam'$.
\end{lem}

\bp By inspection of the expressions of fundamental weights in terms of simple roots \cite[Planches 1-IX]{Bo} one observes that
all these expressions have non-negative rational coefficients. Let $\om_0$ denote the zero weight. By assumption, $\lam-\lam'$ is a radical dominant weight, so $\lam-\lam'=\sum b_i\al_i$, where $b_i\in{\mathbb Q}, b_i\geq 0$. On the other hand, as simple roots are linear independent,  $\lam-\lam'\in \mathcal{R}$ means that all $b_i$ are integers. Therefore, $\lam-\lam'\succ \om_0$. This is equivalent to  $\lam\succeq\lam'$.\enp

The following lemma is well known.

\begin{lem}\label{wr2} Let $\mathbf {G}$ be a simple algebraic group and V a rational $\mathbf {G}$-module. If V has zero weight  then $V^g\neq 0$ for every element $g\in \mathbf {G}$. \el

\begin{thm}\label{z1} {\rm\cite[Theorems 1 and 3]{z90}}\label{z90} Let G be a finite Chevalley group G in defining characteristic $p>0$, and let
 $\St_G$ be the Steinberg representation of G  over the complex numbers. Suppose that $p>2$ or $p=2$
 and G is of type  $G_2(q)$, $E_6(q)$, ${}^2E_6(q)$, $E_8(q)$, $F_4(q)$, ${}^2F_4(q)$, ${}^3D_4(q)$, $C_n(q)$ $(n=4k$ or $4k-1$, $k=1, 2\ldots)$, $D_n(q)$, ${}^2D_n(q)$ $(n=4k$ or
$4k+3 $, $k=1, 2\ldots )$, $A_n(q)$, ${}^2A_n(q)$ $(n$ even). Then for every torus  $T$ of $G$
the trivial representation $1_T$ of $T$  is a constituent in the restriction  of
$\St_G$ to $T$.\end{thm}

%\jc{I switched the Steinberg notation from $\phi$ to $\St_G$.  Is this ok?}

\begin{lem}\label{su9} Let $G = {}^2B_2(q)$. Then the
  $2$-modular Steinberg representation $\St_q$ of G is unisingular.
 Moreover, if $g\in G$ is an odd order element then $\St_q(g)$ has exactly $|g|$ distinct eigenvalues.
\end{lem}

\begin{proof}
In \cite[2.8]{z90},  the second claim is deduced from the character table of $G$ over $\C$. This remains true in characteristic 2 as the  Steinberg representation of $G$ over $\C$ is of 2-defect 0.
\end{proof}

\begin{lem} \label{sl2}
The group $G=\SL_2(q)$ with $q$ even has no non-trivial unisingular $2$-modular irreducible representation.\end{lem}

\begin{proof}
Let $g\in G$ be of order $q-1$. By \cite[Corollary]{KOS}, the irreducible representation  with highest weight $\si_q$ is the only irreducible representation $\phi$ such that $\phi(g)$ has eigenvalue 1.  It is well known that $\St(h)$ does not have eigenvalue 1 for $h\in G$ with $|h|=q+1$ (one can inspect the character table of $G$). This remains true in characteristic 2 as the  Steinberg representation of $G$
over $\C$ is of defect 0.\end{proof}

\bl{ag3}
The group $H = \AGL_2(3)$ has exactly four  $2$-modular \ir representations, whose degrees are $1,2,8,16$. The group $G=\ASL_2(3)$ has exactly four  $2$-modular \ir representations, whose degrees are $1,8,8,8$;  one of those of degree $8$ can be realized over $\F_2$.
\el

\bp
For $H$ and $G$ we use the information from
\begin{center}
\href{https://people.maths.bris.ac.uk/~matyd/GroupNames/432/AGL(2,3).html}{https://people.maths.bris.ac.uk/$\sim$matyd/GroupNames/432/AGL(2,3).html},

\href{https://people.maths.bris.ac.uk/~matyd/GroupNames/193/ASL(2,3).html}{https://people.maths.bris.ac.uk/$\sim$matyd/GroupNames/193/ASL(2,3).html},
\end{center}
respectively. Note that $H$ has 4 conjugacy classes of elements of odd order. As the number of
$2$-modular \ir representations of any finite group equals the number of conjugacy classes of odd order elements, the first claim in the statement follows. Furthermore, every \ir character of $H$ restricted to the odd order elements  is a sum of \ir Brauer characters of $H$. Easy computation based on this fact
reveals that the \ir characters of degree not equal to 2,8,16 are reducible modulo 2, and the characters of the same degree coincide on the odd order elements.

For $G$, we deduce that all \ir 2-modular \reps lift to characteristic 0. The second claim follows
from the fact that one of the Brauer characters of degree 8 takes rational  values.
\enp

\begin{rmk}
From Lemma \ref{ag3} we deduce that each $G\in \{ \AGGL_1(9)$, $\AGL_1(9)$,  $\PSL_3(2)\}$ has a unique $2$-modular \irr of degree $8$ and no \irr of greater degree.  Moreover, the groups
$$
\AGGL_1(9),\  \AGL_1(9),\  \PSL_3(2), \text{and} \ \ASL_2(3)
$$
are the only subgroups of $H$ with a $2$-modular \irr of degree $8$.  This follows from analyzing the maximal subgroup tree of the sites listed above.
\end{rmk}

\begin{rmk}
Note that $\AGL_2(3)\cong \Aut \PSU_3(2) \cong \PSU_3(2):S_3$.
%[[John, this is not indicated on Tim's page, he only write $AGL_2(3)=PSU_3(2):S_3$, but this must be true.]]
The fact that $\PSU_3(2)$ has a unique 2-modular \irr of degree 8 is a special case
of a general result on the Steinberg \rep of an arbitrary group of Lie type. In addition,
if an \irr of a centerless group $H$ is unique, it extends to a projective \rep of $\Aut(H)$.
In our case this is equivalent to an ordinary representation.
\end{rmk}

\begin{lem}\label{au1} Let $G$ be an \ir fixed point subgroup of $\Sp_{2n}(q)$, for $q$ even. Then the following statements hold:
%\begin{enumerate}

$(1)$ $Z(G)=1$;

$(2)$ $G$ has no non-trivial normal $2$-subgroup; moreover, if $G$ is   a subgroup of $X\subset \Sp_{2n}(q)$
then $X$ has no non-trivial normal $2$-subgroup;

$(3)$ every odd order element of $\Sp_{2n}(q)$ is conjugate to its inverse;

$(4)$ every abelian subgroup of $G$ is reducible. %in particular, it is not of order $q^n+1$.
%\end{enumerate}
\end{lem}

\bp Let $V$ be the underlying space of $\Sp_{2n}(q)$.
\begin{enumerate}
\item %Let $G\subset GL_n(q)$ be an irreducible fixed-point subgroup. Then   $ Z(G)=1$.
Suppose the contrary. Let $z\in Z(G)$ and let $V^z$ be the 1-eigenspace of $z$ on $V$. Then $G$ stabilizes $V^z$, a contradiction unless $V^z=V$, but this means that $z=\Id$.
\item Suppose the contrary. Let $S\neq 1$ be a normal  $2$-subgroup of $G$. Then $V^S\neq 0$ and
$V^S$ is $G$-stable, a contradiction.
\item This is well known. In fact, $V$ is self-dual as an $\F_q \Sp_{2n}(q)$-module, and as an $\F_q H$-module for every subgroup $H$ of $\Sp_{2n}(q)$.
\item Suppose the contrary,  let $A$ be an abelian subgroup of $\Sp_{2n}(q)$ and $1\neq a\in A$.
By assumption, $V^a\neq 0$. Then $AV^a=V^a$, a contradiction.  All subgroups of order $q^n+1$
are conjugate. %(see for instance \cite[Lemma 7.1]{EZ}).
\end{enumerate}
\enp

\begin{prop}\label{gf1} Let $G\subset \Sp_{2k}(q)$ be a maximal \ir fixed-point subgroup. Then $H=G\times \Sp_{2n}(q)$ (a ``diagonal" embedding) is a maximal  fixed-point subgroup of $\Sp_{2(k+n)}(q)$.
%unless possibly $k=n=2$.
\end{prop}

\bp Let $V,V_1,V_2$ be the standard modules for $\Sp_{2(k+n)}(q)$, $\Sp_{2k}(q)$ and $\Sp_{2n}(q)$, respectively. By  ``diagonal" embedding we mean the embedding which agrees with $V_1\oplus V_2\rightarrow V$ such that $V_1,V_2$ are non-degenerate subspaces of $V$. Note that $k>1$ by \cite[Ch. I,\S 1, Exercise 1]{serre}, and in fact $k>2$ by \cite{katz} as $Sp_2(q)$ and $Sp_4(q)$ has no \ir fixed-point subgroup.

Suppose the contrary. Let $X$ be a  fixed-point subgroup of $\Sp_{2(k+n)}(q)$ containing $H$. Suppose first that $X$ is irreducible. Then $X$ contains an \ir subgroup $Y$ containing all transvections of $H$, in particular, $\diag(\Id_{2k},\Sp_{2n}(q))\subset Y$.  Irreducible subgroups of $\GL_{2(n+k)}(q)$ generated by transvections
are known: if $q$ is odd then  these are $\SL_{2(n+k)}(q')$, $\SU_{2(n+k)}(q')$, $\Sp_{2(n+k)}(q')$ with $q'|q$. If $q$ is even then, additionally, these are $\OO_{2(n+k)}(q')$, $S_{2n+2}$,  $S_{2n+1}$ or
$A_6\subset \SO^+_4(q)$, see \cite[\S 12]{z81} for detailed references. The latter case is ruled out by assumption. An easy analysis implies
that only $\Sp_{2(n+k)}(q)$ can contain $\diag(\Id_{2k},\Sp_{2n}(q))$, so $Y=\Sp_{2(n+k)}(q)$.
This group is not fixed point. So $X$ is reducible. %Let $V$ be the standard module for $Sp_{2(k+n)}(q)$.
Let $0\neq W\subset V$ be a  $X$-stable subspace of $V$. Then .
 $W=V_i$ for $i=1$ or 2. As $V_i$ is non-degenerate, $X$ stabilizes $V_i^\perp$.
\itf that $X$ stabilizes both $V_1,V_2$. As $\diag(\Id_{2k},\Sp_{2n}(q))\subset X$, we deduce that
the restriction $X_1$ of $X$ to $V_1$ contains $G$ and $X_1 $ must be fixed point on $V_1$.
Then $X_1=G$, and then $X=H$, a contradiction.
\enp

With these preliminary results complete, we now turn to the main business of the paper in the next several sections.

\section{The Affine Groups}

In this  section we show that affine linear groups occur as fixed-point subgroups of certain symplectic groups $\Sp_m(2)$.

An affine group is defined as a semidirect product of $A=\F_q^+$ and a subgroup $H\subset \GL_n(q)$
such that $H$ transitively permutes  the non-zero elements of $A$. Such groups $H$ have been classified by Hering \cite[\S 5]{He}; see also \cite{Dix}. Then $H$ acts 2-transitively on the cosets $AH/H$.
In particular, $AH$ is a 2-transitive subgroup of $S_n$, where $n=|A|$.

Theorem \ref{af1} is a special case of the \f result:

 \begin{thm}\label{af0} Let $AH$ be an affine group with $|A|$ odd,  and $m=(|A|-1)/2$.
  Suppose that $n>1$ or $n=1$ and q is not a prime.
 Then there exists an absolutely \ir fixed-point subgroup of $\Sp_{2m}(2)$ isomorphic to $H$.
\end{thm}

\bp Let $n=|A|$. Then $AH$ is a 2-transitive subgroup of $S_n$. Let $R$ be the natural permutational $FS_n$-module, where $F$ is an \acf of characteristic 2.  As $n$ is odd, we have $R=F\oplus M$, where $M$ is \ir and $F=1_{S_n}$ stands for the trivial $FS_n$-module. This is also true for $\F_2$ in place of $F$. If $M_2$ is the corresponding $\F_2S_n$-module then $S_n$ is well known to preserve a non-degenerate alternating form on $M_2$;
this yields an embedding $\phi:S_n\rightarrow \Sp(M_2)$. As an $FA$-module, $R$ is completely reducible.

Let $M=\sum_{\al\neq 1_A}R_\al$ be the sum of the non-trivial \ir $FA$-submodules of $R$. As  $H$ is transitive on $A\setminus 1$, it also  is transitive on the non-trivial
\ir characters of $A$. \itf $\dim R_\al=1$ for every $\al$ and $M$ is \ir $FAH$-module as well as $M_2$.
%Let $\phi:AH\rightarrow GL(M)$ be the \rep afforded by $M$.

It remains to show that $\phi(AH)$ is a fixed point subgroup. For this we can assume that $H=\GL_n(q)$.
Let $h\in AH$. To prove that 1 is an \ei of $\phi(h)$, it suffices to assume that $|h|$ is odd.

In a permutational module $\C X$-module
of a  group $X$ the \mult of the trivial character $1_X$ equals the number of the $X$-orbits at the
underlying permutation set (see \cite[Theorem 32.3]{CR}. This is true for the $FX$-module
if $|X|$ is odd. Applying this to the cyclic group $\lan h\ran$
 and the permutational set $AH/H$, it suffices to  show that  $\lan h\ran$ has at least 2 orbits on $A$. For this  it suffices to  show that  $|h|<|H:G|=q^n$ for every $h\in H$.
 %(As $|h|$ is odd, this conclusion holds for $F\lan h\ran$-module $M$.)

Note that the  order of an element in $\GL_n(q)$ is at most $q^n-1$ (see for instance \cite[Corollary 2.7]{PS}).  If  $\lan h\ran\cap A=1$ then $|h|$ does not exceed the maximum of element order in $H/A\cong G\subset \GL_n(q)$, so $|h|\leq q^n-1$.

Suppose that  $\lan h\ran\cap A\neq 1$. Let  $h_0$ be a generator of $\lan h\ran\cap A$. Then $h\in C_H(h_0)$,
and hence $|h|$ does not exceed $rx$, where $x$ is the maximum of element orders in $C_G(h_0)$. Note that
$C_G(h_0)$ can be interpreted as the stabilizer in $G$ of a non-zero vector in the underlying vector space of $\GL_n(q)$. So $C_G(h_0)$ is isomorphic to the affine group $\F_q^{n-1}\GL_{n-1}(q)$. We use induction. If $n=2$ then $\F_q^{n-1}\GL_{n-1}(q)=\F_q\GL_1(q)$. If $x\in \F_q\GL_1(q)$ then
$|x|\leq q-1$ unless $x\in \F_q$ and $q=r$. Then $|x|=r$, and hence $|h|$ is an $r$-power.
Note that $\AGL_2(q)$ is isomorphic to a subgroup  $\begin{pmatrix} *&*&*\cr *&*&*\cr 0&0&1\end{pmatrix}$
over $\F_q$, and the subgroup  % or the transpose of it.
%In this group a non-trivial $r$-element is similar to
$\begin{pmatrix} 1&*&*\cr 0&1&*\cr 0&0&1\end{pmatrix}$ is a Sylow $r$-subgroup of it.
%or $\begin{pmatrix} 1&1&0\cr 0&1&1\cr 0&0&1\end{pmatrix}$.
This is of exponent  $r$, as $q$ is odd. \itf $|h|\leq q^2-1$. Let $n>2$.
By induction, the maximum element order in $\F_q^{n-1}\GL_{n-1}(q)$ does not exceed  $q^{n-1}-1$,
so  $|h|\leq r(q^{n-1}-1)<q^n-1$.

Let $n=1$. Then a  maximum element order in $\AGL_1(q)$
 is  $q-1$ unless $q=r$. In the latter case $|h|=r$ for $1\neq h\in A$, and 1 is not an \ei of $h$.
\enp

\begin{cor}\label{se1} Let $d>4$. Then $\Sp_{2d}(2)$ has a fixed point subgroup $G$ of order $27$ and exponent $3$ with no trivial composition factor. If $d=4$ then there is such group $G$ of order
$9$. %the elementary abelian group $G$ of order 27 has a unisingular fixed point
\end{cor}

\bp If $d=4$ then we choose for $G$ the subgroup $A$ of $\AGL_2(3)$ in Theorem \ref{af0}, whose proof
 makes evident that    $A$ has no
trivial composition factor. Let $d>4$ and let $V$ be the underlying space of $\Sp_{2d}(2)$.
Let $V=V_1\oplus V_2\oplus \cdots \oplus  \oplus V_{d-4}$, where $V_1\ld V_{d-4}$ are non-degenerate mutually orthogonal subspaces such that $\dim V_1=8$ and $\dim V_2=\cdots=\dim V_{d-4} =2$.
Let $g\in \Sp_2(2)$ be an element of order 3, and $G=A\times C_3$. The \rep in question
is the sum of a \rep $A\rightarrow \Sp_8(2)$ with $|A|=9$ considered above, and the \rep $C_3 \rightarrow \Sp_{2(d-4)}(2)$, which sends $g $ to the element $\diag(g,g\ld g)$. Then the \rep obtained satisfies
the conclusion of the corollary.\enp

 \begin{rmk} Let $A$ be an elementary abelian group of finite order $m$. The proof of Theorem \ref{af0} shows that $A$ is isomorphic to a fixed point subgroup of $\Sp_{m-1}(2)$ which has no trivial composition factor. In addition, $A$ is a minimal subgroup of  $\Sp_{m-1}(2)$ with this property, that is,  every proper subgroup $B$ of $A$ has a trivial composition factor.\end{rmk}

\bl{ap1} Let G be a group such that $A\subset G\subseteq AH\subset S_n\subset \Sp(M)$. Suppose that
G is \ir on M. Then G acts transitively on $A\setminus 1$, in particular, $|G|$ is a multiple of
$|A|-1$.

If $|A|=9$ then $G/A$ is isomorphic to one of the following groups:  $\GL_2(3),\SL_2(3)$, $D_{16}$, $Q_8$ or  $C_8$, and $G\in\{\AGL_2(3),\ASL_2(3),\AGGL_1(9), \PSU_2(3),\AGL_1(9)  \}$.\el

\bp It is known that $G$ is a 2-transitive subgroup of $S_n$, see for instance \cite[Theorem 3.10]{KS}.
This implies that $G/A$ is transitive on $A\setminus 1$ when $G$ acts on $A$ by conjugation.
This implies $|G|$ to be a multiple of
$|A|-1$.

Let $|A|=9$. Then $G/A$ is a subgroup of $\GL_2(3)$ and 8 divides $|G/A|$. Clearly, $\SL_2(3)$
is transitive on $A\setminus 1$. Other subgroups of $\SL_2(3)$ whose order is a multiple of 8 are 2-groups.
The Sylow 2-subgroup $D$, say, of $\GL_2(3)$ is dihedral of order 16. This is transitive on $A\setminus 1$ as $|C_D(a)|=2$ for $1\neq a\in A$. Let $D_1$ be a subgroup of $D$. Then $|D_1|\leq 8$, we have to inspect $D_1$ of order 8. This is transitive on $A\setminus 1$ \ii $|C_{D_1}(a)|=1$, equivalently, \ii $D_1$ contains no reflection (of Jordan form $\diag(1,-1)$). Clearly, the only such group containing a reflection is the dihedral group $D_8$ of order 8. We can assume that
$$
D_8=\left\{ \begin{pmatrix}\pm 1&0\\ 0&\pm 1\end{pmatrix}, \begin{pmatrix} 0&\pm 1\\ \pm 1&0\end{pmatrix}\right\}.
$$
Moreover, the group $Y=3^2:D_8\cong S_3\wr C_2$ is not \ir in $\Sp_8(2)$.
Indeed, the character table of $Y$, available at {\small
\begin{center}
\href{https://people.maths.bris.ac.uk/~matyd/GroupNames/61/S3wrC2.html}{https://people.maths.bris.ac.uk/$\sim$matyd/GroupNames/61/S3wrC2.html},
\end{center}}
\noindent implies that the \ir 2-modular \reps of it are of degrees 1,4,4, and the Brauer character values are integers.
By Lemma \ref{tt1},  each of these \reps is realized over $\F_2$. Therefore, $Y$ has no \irr of degree 8 over $\F_2$.
The remaining subgroups of order 8 are cyclic $C_8$  and the quaternion group of order 8 (a Sylow 2-subgroup of $\SL_2(3)$).
\enp

\section{Fixed-point subgroups of $\Sp_8(2)$ }\label{applications}

In this section we focus on the smallest symplectic groups $\Sp_{2n}(p)$ for which a complete classification of its irreducible fixed-point subgroups is of yet unknown: the case $n=4$ and $p=2$.  We remind the reader that this is the motivating case from the point of view of 2-torsion on abelian fourfolds. In this section $H=\Sp_8(2)$ and $V$ the symplectic space over $\F_2$ on which $H$ acts in the natural way. So $H$ preserves a non-degenerate bilinear form $(\cdot,\cdot)$ on $V$.
If $U$ is a subspace of $V$ then $U^\perp=\{x\in V: (x,U)=0\}$.
 \medskip

We denote by $\mathcal{F}$ the set of irreducible fixed-point subgroups of $H$. In this section we determine $\mathcal{F}$, that is, we prove Theorem \ref{cmain}. % so we shall proof that $\mathcal{F}$$
We focus on the first statement as the second one follows from Lemma \ref{ap1} and easy analysis of subgroups of $\PGL_2(7)$.

%\begin{proof}[Proof of Theorem {\rm \ref{cullmainthm}}]
Our strategy is to search in the tree of maximal subgroups to determine the maximal irreducible fixed-point groups.  However, due to the size of the group $\Sp_8(2)$, we first make a few observations that reduce the search size.

The elements of $H$ is not fixed point, hence any irreducible fixed-point subgroup must lie in a maximal proper irreducible subgroup. By \cite[Table 8.48]{bhrd}, the proper, \emph{irreducible} maximal subgroups $M$ of $H$ are given in Table 1, in increasing size order.

\med
\begin{center}
\begin{tabular}{|l|l|}
\hline
Maximal Subgroup & Order \\
\hline
${\rm L}_2(17)$ & $2^4.3^2.17$ \\
$\Sp_4(2) \wr S_2$ & $2^9.3^4.5^2$\\
$\Sp_4(4).2$ & $2^9.3^2.5^2.17$\\
%\item $GL_4(2):2\cong S_8;$
$S_{10}$ & $2^8.3^4.5^2.7$\\
$\SO_8^+(2)$ & $2^{13}.3^5.5^2.7$\\
$\SO_8^-(2)$ & $2^{13}.3^4.5.17$\\
\hline
\end{tabular}

\smallskip
\textsc{Table 1:} Maximal Irreducible Subgroups of $\Sp_8(2)$
\end{center}

\med
Until the end of this section, $G\in {\mathcal F}$ and
 $M$ denotes a maximal subgroup of $H$ containing $G$.

\bl{55a}
%Let $G\in {\mathcal F}$. Then
$G$ has no element of order $5, 17$, or $21$.
\el

\bp  The elements of order 17 and 21 of $H$ do not have \ei 1. For the remaining claim, suppose the contrary and let $h\in G$ be of order 5. Observe that $G$ is primitive on $V$. (Otherwise $V=V_1+V_2$, where $V_1,V_2$ are disjoint subspaces permuted by $G$, and  $gV_1=V_2$, $gV_2=V_1$ for some $g\in G$. Clearly,  $hV_i=V_i$ and $h$ is either trivial on $V_i$ or irreducible for $i=1,2$. As  $V^h\neq 0$,  we may assume that  $V^h=V_2$. But then $hghg\up$ acts fixed point freely on $V$, a contradiction.)

 As $G$ is primitive, we  can ignore $M= \Sp_4(2) \wr S_2$.

Let $M=\Sp_4(4).2$. Then $h$ lies in a subgroup isomorphic to $\SL_2(16)$ or
  $(A_5\times A_5):2$ \cite[p. 44]{Atlas}. The latter case is ruled out as above, the former case is ruled out by \cite{JLPW}.

Let $M= S_{10}$. Let $L$ be a maximal subgroup of $M$.
Then $L$ has an element of order 5 \cite[p.49]{Atlas}, but we can ignore
$L$  with $O_2(L)\neq 1$.  Groups   $L\cong S_6\times S_4$ and $S_7\times S_3$ are reducible (by Clifford's theorem, say), and   $L\cong (S_5\times S_5)\cdot 2$ is imprimitive.  We are
left with  $L\cong S_9$ and $L\cong M_{10}.2$. As  $G\subset S_{10}$, the Brauer character $\phi$ of
$G$ has integral values, so $\phi|_{A_6}=\phi_2+\phi_3$ in the notation of \cite{JLPW}, and then $\phi(h)$ is fixed point free by \cite[p. 4]{JLPW}. So $G\subseteq S_9\subset S_{10}$.  \itf $\phi(h)=3$ and then $\phi(x)=-1$ for $x\in S_9$ of order 9, and hence $x$ is fixed point free. So $G\neq S_9$. In addition, $G\neq A_8,S_8$ as otherwise $\phi(g)=-2$ \cite[p. 49]{JLPW}. If $G\subset S_8, G\neq A_8$ and $(5,|G|)=1$ then $G$ is reducible. This is a contradiction.
\enp

\bl{a34} Let   $A$ be a maximal abelian normal subgroup of $G$. If $A\neq 1$ then $A=O_3(G)$, $A=C_3\times C_3$ and $G$ is isomorphic to a subgroup of $\AGL_2(3)$. In particular, this holds if $G$ is solvable.\el

\bp By Lemma \ref{55a}, $(|A|, 85)=1$. Let $B$ be
a Sylow 7-subgroup of $A$. Then $B=\langle b \rangle$ is cyclic and normal in $G$. Then $V^B\neq0$ is $G$-stable, a contradiction.

So $|A|$ is a $3$-power. As above, one observes that $A$ has no element of order 9. So $A$ is elementary abelian, and $|A|\leq 27$ to be fixed point. Suppose that $|A|=27$. In addition,
$(|G|,7)=1$.  (Indeed, if $g\in G$ is of order 7 then $[g,A]\neq 1 $ (otherwise $G$ has an element of order 21, contrary to Lemma \ref{55a}), and $g$ cannot normalizes $A$ as $(|\GL_3(3)|,7)=1$. \itf that
$2,3$ are the only prime divisors of  $|G|$, in particular, $G$ is solvable.
Note that $V^A\neq 0$ is $G$-stable so $V^A=0$; in particular,   $A$ is not cyclic. If $A\neq O_3(G)$
then $O_3(G)$ is non-abelian, and hence there is an \ir constituent of $O_3(G)$ on $V$ whose dimension is a multiple of 3. This is false by Clifford's theorem.

Let $Y$ be a minimal \ir normal subgroup of $G $ such that $A<Y$. Then $Y$ has a reducible subgroup $Y_0$, say, of index 2. (Otherwise,  $G$ has a normal subgroup $G_3$ of index 3 with   $A<G_3$. By Clifford's theorem,   $G_3$ is irreducible, violating the minimality of $Y$.)
By Clifford's theorem, either $V$ is a homogeneous $\F_2Y$-module or $V=V_1+V_2$,  where $V_1,V_2$ are $Y_0$-stable subspaces of $V$. Then $V_1,V_2$ are \ir and non-isomorphic $\F_2Y_0$-modules (this essentially follows from \cite[Theorem 51.7]{CR}). So $V_1,V_2$ are permuted  by $y\notin Y\setminus Y_0$.

 Observe first that (*) if $1\neq a\in A$ then $a$ is non-trivial on $V_2$. (Indeed, suppose that  $V_2\subseteq V^a$. If $a$ fixes no non-zero vector on $V_1$ then $a hah\up$ is fixed point free on $V$, a contradiction. So $V_2\neq V^a$.
 Then $\dim V^a=6$. Note that  $A$ is generated by 3 conjugates of $a$. (If not there is a proper normal subgroup $A_0$ of Y containing $a$. Then $A_0$ is generated by 2 conjugates of $a$. Then $V^{A_0}\neq 0$ is $Y$-stable.)  So let $A=<a,b,c>$, where $a,b,c$ are conjugates of $a$. Then $V^A=V^a\cap V^b\cap V^c\neq 0$ is $Y$-stable, a contradiction.)

Let $a,b,c$ be generators of $A$, and $a_2,b_2,c_2$ the restriction of them to $V_2$. By the above, each of them is non-trivial on $V_2$. Clearly,
the restriction of $A$ to $V_2$ is of rank 2 (due to (*) this is not of rank 1), so we can assume $c_2= a_2^ib_2^j$ for some $i,j$. Then $c\up a^i,b^j\neq 1$ is trivial on $V_2$,  contradicting to (*).

Thus $A=C_3\times C_3$ and $C_G(A)=A$. So $G/A$
is isomorphic to a subgroup $X$ of $\GL_2(3)$ such that $O_3(X)=1$. Then $X$ is either a 2-group or
$\SL_2(3)$ or $\GL_2(3)$. In addition,  $G=A:D$, where $D\cong X$. This is obvious if $X$ is a 2-group, otherwise let $z$ be the central involution in $X$. Then there is an involution $t$, say, of $G$
such that $z=tA$. Then $[t,G]\subset A$ and $tat=a^{-1}$ for every $a\in A$. Observe that
the coset $Ag$ ($g\in G\setminus A)$ meets $C_G(t)$. Indeed, if $tgt=ag$ with $1\neq a\in A$
then $ta^{-1}gt=a^2g=a^{-1}g$. \itf $C_G(t)\cong G/A$, so we can take $D=C_G(t)$. As $\AGL_2(3)\cong A:\GL_2(3)$, the result follows.\enp

\bl{so1} %Let $G\in {\mathcal F}$.
If $(7,|G|)=1$ then $G\in \{\AGL_2(3),\ASL_2(3),\AGGL_1(9)$, $\AGL_1(9)$,  $\PSL_3(2)\}$.
\el

\bp Note that $|H|=2^{16}.3^5.5^2.7.17$. By Lemma \ref{55a}, the primes 5 and 17 are not divisors of $|G|$, and 7 is not by assumption. By Burnside's theorem,  $G$ is solvable as $|G|$ has at most 2 prime divisors. Then the result follows from Lemma \ref{a34}.\enp

These observations allow us to present a streamlined proof of Theorem \ref{cmain}.
%{cullmainthm} and finish this section.
For any computations performed in \textsf{Magma}, we refer the reader to the Appendix for the commands.

\begin{proof}[Proof of Theorem {\rm \ref{cmain}}] %{cullmainthm}}]
As mentioned above $G \in \mathcal{F}$ lies in some maximal irreducible subgroup $M$ of Table 1.  By Lemma \ref{55a}, $G$ has no element of order 5 or 17. In addition, $Z(G)=1=O_2(G)$ by Lemma \ref{au1}.  For the remainder of this proof, $L$ denotes a maximal subgroup of $M$. \\

\noindent \textbf{{Case (1)}} Suppose $M\in \{{\rm L}_2(17),\Sp_4(2)\wr S_2,\Sp_4(4).2\}$. Then $(7,|M|)=1$ and the result follows from Lemma \ref{so1}. \\

\noindent \textbf{Case (2)} If $M=\SO_8^-(2)$ then $L\in \{{\rm L}_3(2):2, \, {\rm L}_2(16):4,  \Omega_8^-(2)\}.$
%the latter is a finite simple group of order 197406720.
\begin{itemize}
\item The group ${\rm L}_3(2):2$ is  irreducible in $M$ (otherwise it is not maximal), and then it is a    fixed-point group, which is   listed in the statement of the theorem.
\item The group ${\rm L}_2(16):4$ has order divisible by 17, and all subgroups
of order coprime to 17 are reducible. % fixed-point subgroups of ${\rm L}_2(16):4$.
\item The group $\Omega_8^-(2)$ has index 2 in $\SO_8^-(2)$ and by \cite[p.~89]{Atlas}, every maximal subgroup of $\Omega_8^-(2)$ has index 2 in some other maximal subgroup of $\SO_8^-(2)$.  Therefore, there are no novel subgroups left to analyze that have not been covered already. \\
\end{itemize}

\noindent \textbf{Case (3)} Let $M=S_{10}$ so that  $L\in\{M_{10}.2,S_5 \wr S_2,S_9,A_{10}\}$.
By Lemma \ref{55a}, $M,L\notin {\mathcal F}$. %We will work through the subgroups of these groups case-by-case.

\begin{itemize}
\item  If $L=M_{10}.2$ or $S_5 \wr S_2$, then $(7,|L|)=1$ and the result follows as above by Lemma  \ref{so1}.

\item  If $L=S_9$, then %$L\notin {\mathcal F}$ as elements of order 9  in $L$ do not have \ei 1.
%So we look at  %By , these are
${\rm AGL}_2(3) $ and $%\text{ and }
 A_{9}$ are the only maximal irreducible subgroups of $L$ \cite{Atlas}. The former is fixed-point.
 The latter  is not fixed-point, so we descent to the maximal irreducible subgroups of it.  By \cite{Atlas}, these are $
{\rm ASL}_2(3), $ and two conjugacy classes of subgroups isomorphic to %$ {\rm L}_2(8).3, \
${\rm L}_2(8).3$ (these are conjugate in $S_9$). The former is not maximal fixed-point subgroup
as ${\rm ASL}_2(3)\subset {\rm AGL}_2(3) . $
%\end{itemize}

%\begin{itemize}
%\item The group ${\rm ASL}_2(3)$ is a fixed-point group and the result follows as above by Lemmas  \ref{so1}.
    %is a subgroup of ${\rm AGL}_2(3)$.
 Each of the two groups isomorphic to ${\rm L}_2(8).3$ has a fixed point free element of order 9, so these groups are  not fixed-point, and a search in {\sf Magma} for irreducible fixed-point subgroups reveals that there are none.

\item %As above,   $A_{10}\notin{\mathcal F}$  and
The maximal irreducible subgroups  of $A_{10}$ are
\begin{center}
 $A_9$, $M_{10}$ and $(A_5 \times A_5):4$.
\end{center}
The group $A_9$ has been covered by the analysis of maximal subgroups of $S_9$, and the other two groups have no element of order 7 and so are ruled out as above. \\
  %Therefore, there are no new cases of maximal, irreducible, fixed-point subgroups in this subcase. This  completes our analysis of the maximal subgroup $S_{10}$ of $\Sp_8(2)$. \\
\end{itemize}

\noindent \textbf{Case (4)}  Let $M=\SO_8^+(2)$ so that  $L\in\{
S_5 \wr S_2, \, S_3 \wr S_4,\, S_9, \, \ \Omega_8^+(2)\}$.  As above, we only have to examine $L$ with $(7,|L|=7)$, which are $S_9$ and $ \Omega_8^+(2)$.

The group $H$ has a unique subgroup (up to conjugacy) isomorphic to $S_9$ \cite{JLPW}. As this group has  already appeared  above in our analysis of the case of $S_{10}$, we can safely ignore it.

To finish the proof, we turn to maximal subgroups $L_1$ of $L=\Omega_8^+(2)$. By Lemma \ref{so1}, we can ignore the subgroups having a non-trivial normal subgroup. This leaves us with the following groups  (up to isomorphism):
$$
L_1 \in \lbrace S_6(2), \, A_9, \, (A_5 \times A_5):2^2, \, (3 \times {\rm U}_4(2)):2 \rbrace.
$$
We may  ignore the group $(A_5 \times A_5):2^2$  since it has no element of order 7 (see Lemma \ref{so1}).  One observes from \cite[p.~85]{Atlas} that  there are three conjugacy classes of groups  isomorphic to  each of the remaining groups. Moreover, exactly one of them has index 2 in a maximal subgroup of $\SO_8^+(2)$, so that one is covered by the above analysis. Furthermore, in each case two remaining classes fuse in  $\SO_8^+(2)$, so it suffices to examine any one of these classes.  We now again proceed case-by-case.

\begin{itemize}
\item Let $L_1\cong \Sp_6(2)$.
  If $g\in L_1$ is of order  9 and $\phi$   is a  2-modular \irr of $L_1$ of degree 8 then  $\phi(g)$  does not have \ei 1 \cite[p. 110]{JLPW}. So $G$ has no element of order 9. All maximal subgroups of $L_1$ with this property are reducible in $H$.

 \item
  Let $L_1\cong A_9$. By the above, we only have to consider %Then $L$ has no 2-modular non-trivial \irr of degree less than 8, and three non-equivalent \irr of degree 8
  the 2-modular \irr of $A_9$ denoted by $\phi_3$   in \cite[p. 85]{JLPW}.
  %Then one deduces that 1 is not an \ei of $\phi_3(g)$  for $ |g|=5$.
  We use \textsf{Magma} to determine the maximal subgroups of $K$ of $A_9$ subject to condition $(5,|K|)=1$.  This shows that $K$ is conjugate either to $\PSU_3(2)$ or to $\ASL_3(2)$. These are not maximal fixed point subgroups of $\Sp_8(2)$.
  \end{itemize}

This complete our analysis and the proof of Theorem \ref{cmain}. % {cullmainthm}.
\end{proof}

\section{Sufficient Conditions of Unisingularity}

\subsection{Symplectic and orthogonal groups}

\begin{lem} \label{spo}
Let   ${\mathbf G}$ be the algebraic group of type $C_n, n>1$ or $D_n, n>3$.
 %$V$  the natural module for ${\mathbf G}$, that is, the one with highest weight  $\om_1$.
 Let $G=C_n(q)$ or $D_n^\pm(q)$ and let $g\in G$ a semisimple element. Suppose that  $|g|$ does not divide $q-1$ or $q+1$. Then g is conjugate in ${\mathbf G}$ to an element $t\in {\mathbf G}$  such that
 $\ep_1(t)^q\cdot \ep_2(t)=1$.
\end{lem}

\begin{proof}
This is well known, and explained in detail in many sources, in particular see \cite[\S 3.2]{z16} or \cite[\S 2.4]{HZ09}.
\end{proof}

\begin{prop}\label{aa1}
Let $\mathbf{G}$ be a simple algebraic group of type $C_n$, $n>1$, $(n,q)\neq (2,2)$ or $D_n$, $n>3$. Let
$G=\Sp_{2n}(q)\subset \Sp_{2n}(F)=\mathbf{G}$ or $D^+_{2n}(q),D^-_{2n}(q)\subset D_{2n}(F)=\mathbf{G}$.
Let V be an \ir $\mathbf{G}$-module with \hw $\om$. Suppose that $ m_1(q+1)\om_1, m_2(q-1)\om_1$ and $m_3((q-1)\om_1+\om_2)$ are weights of V  for some natural numbers  $m_1,m_2,m_3$. Then  every semisimple element $g\in G$ has \ei $1$ on V.
\end{prop}

\begin{proof} Let $t$ be as in Lemma \ref{spo}.  Suppose that $|g|$ divides $q-1$ or $q+1$.
%As $q$ is even, either $q-1$ or $q+1$ is a multiple of $|g|$.
In the latter case we have $(m_1(q+1)\om_1)(t)=(m_1(q+1)\ep_1)(t)=(\ep_1(t))^{m_1(q+1)}=1$.
 So the result follows. The case of $q-1$  a multiple of $|g|$ is similar.

Suppose that $|g|$  divides neither $q-1$ nor $q+1$. By Lemma \ref{spo},  we can assume that $\ep_1(g)^q= \ep_2(g)=1$. As $m_3((q-1)\om_1+\om_2)=m_3(q\ep_1+\ep_2)$,
we have   $(m_3(q\ep_1+\ep_2))(g)=(\ep_1(g))^q\cdot \ep_2(g))^{m_3}=1$.\end{proof}

 \begin{lem} \label{rr1}
Let $\Lambda$ be the weight system of type $C_n$, $n>1,$ $(n,q)\neq (2,2)$, or $D_n,n>3$   and let $\si_q=(q-1)(\om_1+...+\om_n)$, $q$ even. Suppose that $\si_q$ is not radical. Then $\si_q\succ (q+1)\om_1\succ(q-1)\om_1+\om_2\succ (q-1)\om_1$.\end{lem}

\begin{proof}
As $\si_q$ is not radical, we have $\si_q\succ \om_1$ for both the groups. (For $C_n$ see \cite[Ch. VIII, \S 7.3]{Bo8}. Let $G$ be of type $D_n$. One observes that either weight 0 or $\om_1$ is a subdominant weight of each of the weights $\om_1\ld,\om_{n-2}, \om_{n-1}+\om_n$. So we conclude similarly.) In addition,
if $\Lambda$ is of type $D_n$ then $n>4$ as  $\si_q$ is not radical. So $\om_1\prec\om_3$ \cite{Bo}. If
$\Lambda$ is of type $C_3$ or $C_4$ then $\si_q$ is radical, so we have $n>4$ again.

 As $q$ is even, $\si_q-(q+1)\om_1\in {\mathcal R}$.  For $q>2$ we have $\si\succ \si_1:=(q+1)\om_1+(q-1)\om_2+(q-3)\om_3+(q-1)(\om_4+...+\om_n)$ as $\om_3\succ\om_1$. In addition,
 $\si_1-(q+1)\om_1\in {\mathcal R}$ is dominant. By Lemma \ref{mm2},
 $\si_1\succ  (q+1)\om_1$.     Whence the first inequality.
And $(q+1)\om_1-\al_1=(q-1)\om_1+\om_2\succ (q-1)\om_1$, as $\om_2$ is radical.

Let $q=2,n>2$. Then $n\geq 5$.  If $n=5$ and 6 then $\si_2\succ 3\om_1\succ \om_1+\om_2$ for both the groups (in the $D_n$-case use $\om_{n-1}+\om_n\succ \om_1$).
If $n\geq 7$ then $\si_2-\om_1-\om_3-\om_5 $ is a radical dominant weight, and $\om_1-\om_3-\om_5\succ 3\om_1$. This implies $\si_2\succ 3\om_1$ by Lemma \ref{mm2}.

Let $n=2,q>2$. Then $\si_q-\al_2=(q+1)\om_1+(q-3)\om_2\succ (q+1)\om_1$ as $\om_2$ is radical, and $(q+1)\om_1\succ (q-1)\om_1+\om_2$.\end{proof}

 \begin{lem}\label{spo1}
 Theorem {\rm \ref{st1}} is true  for $G=C_n(q), n>1$ and $D^\pm_n(q), n>3$. \end{lem}

 \begin{proof} %Let $V$ be the \irr of $G$ with highest weight $\si_q$.
 The case with $G=C_2(2)$ follows by inspection
 of the Brauer character table of $G$ \cite{JLPW}. Assume that $(n,q)\neq (2,2)$. Then $\St_q$ extends to a \ir \rep $\Phi_q$
 of the simple algebraic group $\mathbf{G}$ of type $C_n$ or $D_n$ \cite[Theorem 43]{St}. The highest weight of $\Phi_q$ is $\si_q$. By Lemma  \ref{st2},  the weights of $\Phi_q$
 are the same as in characteristic 0. By  \cite[Ch. VIII, \S7, Prop. 5(iv)]{Bo}, if $\rho,\mu$ are dominant weights, then $\rho\succ\mu$ implies $\mu$ to be a weight of $V$. In particular, $ (q+1)\om_1, (q-1)\om_1$ and $(q-1)\om_1+\om_2$ are weights of $V$ by Lemma \ref{rr1}.
 So the result follows from Proposition \ref{aa1}.
 \end{proof}

\subsection{The groups $\SL_n(q)$}

  In this section $G=\SL_{n+1}(q)\subset \GG=\SL_{n+1}(F)$, $n>1$,   where $F$ be an \acf of characteristic $r> 0$ and $q$ is an $r$-power. Let $\mathbf{T}$ be a maximal torus of $\GG$ whose rational \ir \reps are weights of $\GG$. By $W$ we denote the Weyl group of $\GG$. The conjugacy classes of maximal tori of $G$  are labeled by a partition $\pi=[n_1\geq \cdots\geq n_k]$ of $\{1\ld n+1\}$ \cite[3.23]{DM}, and we can write $T=T_\pi$ to denote a representative of the conjugacy class labeled by $\pi$.
%, where $n_1+\cdots +n_k=n+1$ and $n_1\geq ...\geq n_k$.
For $i=1\ld k$ let $\zeta_j\in \overline{\F}_q$ be a primitive $(q^{n_j}-1)$-root of unity.
It is well known, and explained in detail in \cite[Section 3]{z16},  that $T$ is conjugate in $\GG$ to a subgroup of the diagonal matrix group
 $D_\pi=\diag(D_1\ld D_k)$, where $D_j$ is a cyclic group of order $q^{n_j}-1$ for $j=1\ld k$. More precisely, $D_j=\lan d_j\ran$, where $d_j=\diag(\zeta_j,\zeta_j^q\ld \zeta_j^{q^{n_j-1}})$. If $\phi$ is a rational \rep of $\GG=\SL_{n+1}(F)$ then
 $\phi|_{T_\pi}$ is equivalent to $\phi|_{D_\pi}$ when we identify $D_\pi$ with $T_\pi$.
 As the weights $\om$ of $\GG$ are the one-dimensional \ir rational \reps of $\mathbf{T}$, it is meaningful to write $\om|_{D_\pi}$, even if $\om$ is not a weight of $\phi$.

\begin{prop}{\rm \cite[Prop. 4.1]{z18}} \label{p31}
    Let     $m>1$ be an integer such that
     r is coprime to $(m^{n+1}-1)/(m-1)$. Let $\zeta\in F$ be a primitive $\big((m^{n+1}-1)/(m-1)\big)$-root of unity and set $t_m=\diag(\zeta,\zeta^{m},\zeta^{m^2}\ld \zeta^{m^{n}})\in\GG$.     Let $\mu=\om_1+(m-1)\om_i+\om_n$ $(1\leq i\leq n)$ be the weight of $\GG$. Then $g(\mu)(t_m)=1$ for some $g\in W$. \end{prop}

Next we mimic the reasoning for $q=2$ in  \cite[\S 7]{z18}. %Let $G=\SL_{n+1}(q)$.
%Recall that the $G$-conjugacy classes of the maximal tori of $G$ are in bijection with the partitions $[n_1\ld n_k]$ of $\{1\ld n+1\}$, where $n_1\geq n_2\geq \cdots \geq n_k$ for various integers $k$, $0<k\leq n+1$.

\begin{prop}\label{n10}
Let $T=T_\pi$ be a maximal torus of G  corresponding to a partition $\pi=[n_1\ld n_k]$.
Let $\lam_i=(q-1)\om_i$ for $1\leq i\leq n$.
Suppose that  $i=n_{j_1}+\cdots+n_{j_l}$ for some subset $\{j_i\ld j_l\}$ of  $\{1\ld k\}$. Then the W-orbit of  $\lam_i$ has a weight $\mu$ such that $\mu|_T=1_T$.
\end{prop}

\begin{proof}   Let $W\lam_i$ be  the $W$-orbit of $\lam_i$.
By \cite[Corollary 3.7]{z16},   the number of weights $\mu\in W\lam_i$ such that $\mu|_T=1_T$ equals
$1_{Y}^{S_{n+1}}(\pi)$, where $W\cong S_{n+1}$ is the Weyl group of $\GG$ and $Y\cong S_i\times S_{n+1-i}$ is the Young subgroup of $S_{n+1}$ labeled by $[i, n+1-i]$.
Then  $1_{Y}^{S_{n+1}}(\pi)> 0$ whenever $\pi$ is conjugate to an element of $Y$. This happens \ii $i$ is the sum of some parts of $\pi$, as stated. \end{proof}

\begin{lem}\label{n16}
 Let %$G=\SL_{n+1}(q)\subset\GG=\SL_{n+1}(F)$ and let
 $\phi$ be an \irr of  $\GG$
with highest weight $\om $. Let T be a maximal torus of G. For a fixed $i=1\ld n$
set $\kappa_i=(q-1)\om_i+\om_1+\om_n$ and $\lam_i=(q-1)\om_i$, and let $W\kappa_i$, $W\lam_i$ be the W-orbits of these weights.
Then $\mu|_{D_\pi}=1_{D_\pi}$ for some weight $\mu\in (W\kappa_i\cup W\lam_i)$.
%is a constituent of the restriction $\phi|_T$ for every torus of  $G$.
\end{lem}

\begin{proof}
As explained above, $D_\pi=\diag(D_1\ld D_k)$, where $D_j$ is a cyclic group of order $q^{n_j}-1$ for $j=1\ld k$. More precisely, $D_j=\lan d_j\ran$, where $d_j=\diag(\zeta_j,\zeta_j^q\ld \zeta_j^{q^{n_j-1}})$ and $\zeta_j$ a primitive $(q^{n_j}-1)$-root of unity in $F$.
One observes that $(q\ep_u)(x)=\ep_{u+1}(x)$ for every $x\in D$ and $u\in\{1\ld n+1\}$,  unless $u=n_1+\cdots +n_j $ for $j\in\{1\ld k\}$. In addition, $\det d_1=(\ep_1+\cdots +\ep_{n_1})(d_1)=\zeta_1^{1+q+\cdots +q^{n_1-1}}=\zeta_1^{(q^{n_1}-1)/(q-1)}\in \FF_q$ as this is an element of order $q-1$ in $F$. Similarly, for other matrices $d_2\ld d_k$.

First observe that there exists a subset $\{n_{j_1}\ld n_{j_l}\}$ of   $\{n_{1}\ld n_{k}\}$ such that
$m:=i-1-n_{j_1}-\cdots -n_{j_l}<n_r$ for $r\in \{n_{1}\ld n_{k}\}$ and $r\neq j_1\ld j_l$. (This is trivial as $ n_1+\cdots +n_k=n+1>i$. Note that this subset may be empty, and then $m=i-1 $.)

Suppose first that $m+1=n_r$ for some $r$. Then $i=n_{j_1}+\cdots +n_{j_l}+n_r$.
  So in this case the result follows  from Lemma \ref{n10}.
Therefore, we now assume that $m+1<n_r$ for every $r$.

Let we fix these subset $n_{j_1}\ld n_{j_l}$, and define a subgroup $D'\in \GG$ by moving
  $D_{j_1}\ld D_{j_l}$ to the positions $k+1-l\ld k$ and  the submatrices $D_u$ for $u\in \{k-l+1\ld k\}$  and $u\neq j_i\ld j_l$ to the positions prior
to  $ k+1-l$.  Set $n'=n_{j_1}+\cdots+ n_{j_l}$ so $i=m+1+n'$.  Clearly, $D$ and $D'$ are conjugate in $\GG$ so it suffices to prove that there is a weight $\mu$ of $\phi$ such that $\mu(D')=1$.

   Note that $\kappa_i=(q-1)\om_i+\om_1+\om_n=(q+1)\ep_1+q\ep_2+\cdots +q\ep_i+\ep_{i+1}+\cdots +\ep_n$. The $W$-orbit of $\kappa_i$ contains the weight $\mu=(q+1)\ep_1+q\ep_2+\cdots +q\ep_{m+1}+ \ep_{m+3}+\cdots +\ep_{n+1-n'} +q\ep_{n+2-n'}+\cdots +q\ep_{n+1}
$. Then $\mu(D')=((q+1)\ep_1+q\ep_2+\cdots +q\ep_{m+1}+ \ep_{m+3}+\cdots +\ep_{n+1-n'}
)(x)\cdot (q\ep_{n+2-n'}+\cdots +q\ep_{n+1})(x)$ for $x\in D'$.

Consider the second multiple. %Indeed, %if $x\in D'$ then
%as if $x=\diag(x_{j_1}\ld x_{j_l})$ with $x_{j_u}\in D_u$ for $u=1\ld l$ then
Here $(\ep_{n+2-n'}+\cdots +\ep_{n+1})(x)$ is the product of the diagonal entries of certain matrices $d_j$ above, and hence this lies in $\FF_q$. Therefore, $(q\ep_{n+2-n'}+\cdots +q\ep_{n+1})(x)=(\ep_{n+2-n'}+\cdots +\ep_{n+1})(x)$.

 Consider the first multiple. As $n_r>m+1$ for every $r\neq j_1\ld j_l$, $r\in \{1\ld k\}$,  we have
$((q+1)\ep_1+q\ep_2+\cdots +q\ep_{m+1})(D')=(\ep_1+(q+1)\ep_2+q\ep_3+\cdots +q\ep_{m+1})(D')=(\ep_1+\cdots +\ep_{m+1}+\ep_{m+2})(D')$. Therefore, $((q+1)\ep_1+q\ep_2+\cdots +q\ep_{m+1}+ \ep_{m+3}+
 \cdots +\ep_{n+1-n'})(x)=(\ep_1+\cdots +\ep_{n+1-n'})(x)$. (To make it more clear, write $D'=\diag(D_1'\ld D_k')$.
Let $n_j'$ be the size of the matrix $D_j'$ $(j=1\ld k)$. Then $n_1'>m+1$ by the above, and $q\ep_u(x)=\ep_{u+1}(x)$ for $u=1\ld n_1'-1$ and any $x\in D'$.)

Finally, $((\ep_1+\cdots +\ep_{n+1-n'})(x))\cdot ((\ep_{n+2-n'}+\cdots +\ep_{n+1})(x))=\det x=1$.
\end{proof}

\begin{rmk} If $\phi$ is $p$-restricted then $(q-1)\om_i$ is a weight of $\phi$ whenever so is $(q-1)\om_i+\om_1+\om_n$. In general, this is not probably true. The following special case
illustrates the situation.\end{rmk}

\begin{lem}\label{ww2}
Let $\phi$ be an \irr of $\GG$ with highest weight $ (q-1)\om_i+\om_1+\om_n$. Then $(q-1)\om_i$
is a weight of $\phi$.\end{lem}

\begin{proof} If $q$ is a prime then the result follows from Premet's theorem \cite[p.23]{Hu1}. Let $q>p$ be a $p$-power.
Then $\phi=\phi_1\otimes\phi_2$, where $(p-1)\om_i+\om_1+\om_n$, $(q-p)\om_i$ are the highest weights of $\phi_1,\phi_2$, respectively. The weights of a tensor product are $\lam+\mu$ with $\lam,\mu$ to be weights of $\phi_1,\phi_2$, respectively. By Premet's theorem, $(p-1)\om_i$ is a weight of $\phi_1$ so $ (q-1)\om_i=(p-1)\om_i+(q-p)\om_i$ is a weight of $\phi$.
\end{proof}

\bp[Proof of Theorem {\rm \ref{s21}}] Clearly, $W(m_1\kappa_i)=m_1(W\kappa_i)$
and $W(m_2\lam_i)=m_2(W\lam_i)$. Therefore, it follows from Lemma \ref{n16} that
$\nu|_{D_\pi}=1_{D_\pi}$ for a suitable weight $\nu\in W(m_1\kappa_i)\cup W(m_2\lam_i)$.
As $\nu$ is a weight of $V$ and $D_\pi$ is $\mathbf{G}$-conjugate to $T_\pi$, the result follows.
\enp

\begin{lem}\label{sts} Theorem {\rm \ref{st1}} is true if $G=\SL_{n+1}(q)$.\end{lem}

\begin{proof} By Lemma \ref{mm2}, we have $\si_q\succ  (q-1)\om_{(n+1)/2}+\om_1+\om_n\succ (q-1)\om_{(n+1)/2}$, so the weights  $(q-1)\om_{(n+1)/2}+\om_1+\om_n$ and $ (q-1)\om_{(n+1)/2}$ are weights of  $\Phi_{q}$. Then the result follows from Lemma \ref{n16}.\end{proof}

\bp[Proof of Theorem {\rm \ref{st1}}] The result follows from Lemmas \ref{spo1} and \ref{sts}
for groups $C_n(q)$, $D^\pm_n(q)$ and $A_n(q)$, respectively, and from Lemma  \ref{su9}  for $G={}^2B_2(q)$. For the remaining groups non excluded in Theorem \ref{st1} are contained in Theorem \ref{z1}.\enp

\begin{proof}[Proof of Corollary {\rm \ref{avcor}}]
Let $G$ be as in the statement of the Corollary and let $A$ be an abelian variety defined over $\Q$ of dimension $m/2$ such that $\im \overline{\rho}_2 = \Sp_{m}(2)$.  Extend scalars to $K = \Q(A[2])^G$, the fixed of $G$ in the 2-division field of $A$.  Then, over $K$, $\im \overline{\rho}_2 = G$; the Corollary follows because $G$ is fixed-point and absolutely irreducible.
\end{proof}

\noindent \textbf{Acknowledgements.} We would like to thank Gunter Malle for helpful correspondence.

\appendix

\section{{\sf Magma} Code} \label{appendix}

Many of the computations in this paper were performed using the computer algebra {\sf Magma} at the online calculator \url{http://magma.maths.usyd.edu.au/calc/}.  In this appendix we share the code we used, together with some sample calculations to show the reader how we arrived at some of our conclusions.

To initiate the group $\Sp_8(2)$ and store its maximal subgroups, we use: \\

\begin{tabular}{l}
\texttt{G:=Sp(8,2);} \\
\texttt{M:=MaximalSubgroups(G);} \\[10pt]
\end{tabular}

To run through the maximal subgroups of $G$ and check whether any are fixed-point groups and what the dimensions of the simple factors are, we use the following double loop: \\

\begin{tabular}{l}
\texttt{for i:=1 to \#M do;}\\
\begin{tabular}{l}
\texttt{\#M[i]\`{ }subgroup;} \\
\texttt{C:=ConjugacyClasses(M[i]\`{ }subgroup);}\\
\texttt{for j:=1 to \#C do;} \\
\begin{tabular}{l}
\texttt{CC:=Coefficients (CharacteristicPolynomial(C[j][3]));} \\
\texttt{print C[j][1],"     ",\&+ [CC[j] : j in [1..\#CC]];} \\
\end{tabular} \\
\texttt{end for;} \\

\texttt{A:=MatrixAlgebra<GF(2), 8 | Generators(M[i]\`{ }subgroup)>;}\\
\texttt{MM:=RModule(A);}\\
\texttt{B:=CompositionFactors(MM);}\\
\texttt{B;}\\
\end{tabular} \\
\texttt{end for;} \\[10pt]
\end{tabular}

We remark that we are checking the fixed-point condition by simply summing the coefficients of the characteristic polynomial and checking if it is 0. This procedure can be iterated to check the maximal subgroups of the maximal subgroups, etc. In addition, we make use of the commands  \\

\begin{tabular}{l}
\texttt{L:=LowIndexSubgroups(G,<n,m>);} \\
\texttt{S:=Subgroups(G:OrderEqual:=N);} \\[10pt]
\end{tabular}

\noindent to create lists of subgroups whose index lies in the range $[n,m]$, or to enumerate all subgroups of order $N$, respectively.


\begin{thebibliography}{99}

%\bibitem{BI} E. Bannai  and T. Ito, Algebraic combinatorics, Benjamin, London, 1984.

%\bibitem{BZ} A. Baranov and A. Zalesski, Weight zero in tensor-indecomposable \ir \reps of simple algebraic groups, arXiv:2004.05012v1 [math.RT] 10 April 2020.

\bibitem{lmfdb} A.~Booker, J.~Sijsling, A.~Sutherland, J.~Voight, D.~Yasaki. A database of genus-2 curves over the rational numbers, {\it LMS J. Comput. Math}. 19 (2016), suppl. A, $235-254.$

\bibitem{Bo} N. Bourbaki, {\it Groupes et algebres de Lie}, ch. IV-VI, Masson, Paris, 1981.

\bibitem{Bo8} N. Bourbaki, {\it Groupes et algebres de Lie}, ch. VII-VIII, Springer, Berlin, 2006.

\bibitem{bhrd} J.N.~Bray, D.F.~Holt, C.M.~Roney-Dougal. {\it The maximal subgroups of the low-dimensional finite classical groups. With a foreword by Martin Liebeck}. London Mathematical Society Lecture Note Series, 407. Cambridge University Press, Cambridge, 2013.

%\bibitem{bm} G.~Butler and J.~McKay, The transitive groups of degree up to eleven. \textsl{Comm. Algebra.} \textbf{11} (1983), $863-911$.

\bibitem{Atlas} J. H. Conway, R. T. Curtis, S. P. Norton,  R. A. Parker and R. A. Wilson,
{\it An ATLAS of Finite Groups}, Clarendon Press, Oxford, $1985$.

\bibitem{sp6} J.~Cullinan, Local-global properties of torsion points on three-dimensional abelian varieties. \textsl{J.~Algebra.} \textbf{311} (2007), $736-774$.

\bibitem{2tors} J.~Cullinan, A computational approach to the 2-torsion structure of abelian threefolds. \textsl{Math. of Comp.} \textbf{78} (2009), $1825-1836$.

\bibitem{smallchar} J.~Cullinan, Points of small order on three-dimensional abelian varieties; with an appendix by Yuri Zarhin. {\it J. Algebra} \textbf{324} (2010), no. 3, $565-577.$

\bibitem{gl3} J.~Cullinan, Fixed-point subgroups of $\mathrm{GL}_3(q)$. Journal of Group Theory \textbf{22} (5) $893 - 914$ (2019)

\bibitem{cy} J.~Cullinan, J.~Yelton. Divisibility properties of torsion subgroups of abelian surfaces. Submitted 2020. \url{https://arxiv.org/abs/1912.02356}

\bibitem{CR} Ch. Curtis and I. Reiner, Representation theory of finite Groups and associative algebras,
Interscience, New York, 1962.

\bibitem{DM}  F. Digne and J. Michel, ``{\it Representations of Finite Groups of Lie Type}", London
Mathematical Society Student Texts $21$, Cambridge University Press, $1991$.

%\bibitem{DZ1}  L. Di Martino and A. Zalesskii, Minimum polynomials and lower bounds for eigenvalue
%multiplicities of prime-power order elements in representations of  quasi-simple groups,
%{\it J. Algebra} {\bf 243} $(2001)$, $228 - 263$.
%Corrigendum:   {\it J. Algebra}
%{\bf 296} $(2006)$, $249 - 252$.

%\bibitem{Di}  J. Dieudonne, {\it La g\'eom\'etrie des groupes classiques}, Springer, Berlin, 1971.

\bibitem{Dix} J. Dixon and B. Mortimer, {\it Permutation groups}, Springer-Verlag, Berlin, $1996$.

\bibitem{Do} T. Dokchitser, Structure and character tables of selected finite groups, Project "Groups Names",   https://people.maths.bris.ac.uk/$\sim$matyd/GroupNames/index.html

\bibitem{Fe}  W. Feit, {\it The Representation Theory of Finite Groups}, North-Holland, Amsterdam, $1982$.

\bibitem{GT} R. Guralnick and Pham Huu Tiep, Finite simple unisingular groups of Lie type, {\it J. Group Theory} 6(2003), $271 -310.$

\bibitem{PS} S. Guest, J. Morris, Ch. E. Praeger, and P. Spiga, On the maximum orders of elements of finite almost simple groups and primitive permutation groups, {\it Trans. Amer. Math. Soc}.  367(2015), $7665 - 7694$.

\bibitem{He} Ch. Hering,  Transitive linear groups and linear groups which contain
irreducible subgroups of prime order, {\it Geom. Dedic.} 2(1974), $425-460.$

\bibitem{HZ09}  G. Hiss and A.E. Zalesski, The Weil-Steinberg character of finite classical groups,
{\it Representation Theory} 13(2009), $427 - 459.$

\bibitem{Hu} J. Humphreys, The Steinberg representation, {\it Bull. Amer. Math. Soc.}  16 (1987), $247 - 263.$

\bibitem{Hu1} J. Humphreys, {\it Modular \reps of finite groups of Lie type}, Cambridge Univ. Press, Cambridge, 2005.

\bibitem{JLPW}  C. Jansen, K. Lux, R. A. Parker, and R. A. Wilson, `{\it An ATLAS of
Brauer Characters}', Oxford University Press, Oxford, $1995$.

\bibitem{katz} N.M.~Katz, {\it Galois properties of torsion points on abelian varieties}, {\it Inv.\ Math.} 62(1981), $481-502.$

%\bibitem{KL}  P. B. Kleidman and M. W. Liebeck, `{\it The Subgroup Structure of the
%Finite Classical Groups}', London Math. Soc. Lecture Note Ser. no.
%$129$, Cambridge University Press, $1990$.

%\bibitem{KMT} A. Kleshchev, L. Morotti and Pham Huu Tiep, Irreducible restrictions of \reps of symmetric groups in small characteristic: reduction theorems, Math. Z. 293(2019), 677 - 723.

%\bibitem{KMT2} A. Kleshchev, L. Morotti and Pham Huu Tiep, Irreducible restrictions of \reps of symmetric groups in small characteristic, Adv. Math. (2020), 677 - 723.

\bibitem{KS} A. S. Kleshchev, J. K. Sheth,
Representations of the symmetric group are reducible over simply transitive subgroups,  {\it Math. Z.} 235(2000), $99-109.$

\bibitem{KOS} A.S. Kondratiev, O. Osinovskaia and I.D. Suprunenko, On the
behavior of prime order elements of a Zinger
cycle in representations of special linear group, Proc. Steklov Inst. Math. 285(2014), Suppl.1, $S108-S115.$

\bibitem{serre} J.-P.~Serre, \emph{Abelian $\ell$-adic representations and elliptic curves},
Benjamin, New York, 1068.
%Res. Notes Math., vol.~7, A K Peters, Ltd., Wellesley, MA, 1998.

\bibitem{spr} T. Springer, Characters of special groups, In: "A. Borel et al, {\it Seminar on algebraic groups and related finite groups, Lecture notes in Math.}, vol. 131, Springer-Verlag, Berlin, 1970",
Chapter D, pp. D-1 $-$ D-46.

\bibitem{St} R. Steinberg, {\it Lectures on Chevalley groups}, Amer. Math. Soc. Univ. Lect. Series, vol. 66, Providence, Rhode Island, 2016.

%\bibitem{TZ8}  Pham Huu Tiep and A.E. Zalesski, Hall-Higman type theorems for semisimple elements
%of finite classical groups. {\it Proc. London Math. Soc}. (3) 97(2008), 623 -- 668.

\bibitem{VZ}  R. Vincent and A.E. Zalesski, Non-Hurwitz classical groups, {\it LMS J. Comput. Math.} 10(2007), $21 - 82.$

\bibitem{z81}  A.E. Zalesski, Linear groups,   {\it Russian Math. Surveys} 36(1981), no.5,
      $63 - 128$.

\bibitem{z90}  A.E. Zalesski, The eigenvalue $1$ of matrices of complex representations of finite Chevalley groups, Proc. Steklov Inst. Math. 1991, issue 4, $109 - 119.$

%\bibitem{Za-95} A.E. Zalesski\u\i, Minimal polynomials and
%eigenvalues of $p$-elements in representations of groups with a cyclic Sylow $p$-subgroup, {\it J. London Math. Soc}. (2) 59(1999), $845 - 866.$

%\bibitem{z05}  A.E. Zalesski, The number of distinct eigenvalues of elements in finite linear
%groups. {\it J. London Math. Soc}. Part 2, 74(2006), $361 - 378$.

%\bibitem{z08}  A.E. Zalesski, Minimal polynomials of the elements of prime order in complex irreducible representations of quasi-simple groups. {\it J. Algebra} 320(2008), $2496 - 2525$.

\bibitem{z09}  A.E. Zalesski,  On eigenvalues of group elements in \reps of simple
algebraic groups and finite Chevalley groups, {\it Acta Appl. Math.} 108(2009), $175 - 195$.

 \bibitem{z16}  A.E. Zalesski, Invariants of maximal tori and unipotent constituents of some quasi-projective characters for finite classical groups, {\it J. Algebra} 500(2018), $517 - 541$.

\bibitem{z18}  A.E. Zalesski, Singer cycles in 2-modular \reps of $GL_n(2)$,  {\it Archiv der Math.} 110(2018), $433 - 446$.


\bibitem{z20}  A.E. Zalesski, Abelian subgroups and semisimple elements in $2$-modular representations of the symplectic group $Sp_{2n}(2)$, arXiv:3114388 [math. GR] 20 April 2020.


\end{thebibliography}
\end{document}